\documentclass[12pt]{amsart}
\usepackage{amssymb}
\usepackage{amsfonts}
\usepackage{amsmath}%
\usepackage{graphicx}
\usepackage[export]{adjustbox}
\usepackage{hyperref}
\usepackage{fancyhdr}
\usepackage{caption}
\usepackage{subcaption}
\usepackage{xcolor}
\usepackage{tikz-cd}
\usepackage{float}
\usepackage[
top    = 2.6cm,
bottom = 2.5cm,
left   = 2.5cm,
right  = 2.5cm]{geometry}
\renewcommand{\headrulewidth}{1pt}
\renewcommand\headrule{%
\nointerlineskip
\hbox to \headwidth{\hss\rule{1\headwidth}{\headrulewidth}\hss}%
\vspace{-\headrulewidth}}
\fancyhfoffset[LE]{0pt}
\fancyhfoffset[RO]{0pt}
\makeatletter
\providecommand{\U}[1]{\protect\rule{.1in}{.1in}}
\newtheorem{theorem}{Theorem}

\newtheorem{claim}{Claim}

\newtheorem{conjecture}{Conjecture}
\newtheorem{corollary}{Corollary}

\newtheorem{definition}{Definition}

\newtheorem{lemma}{Lemma}

\newtheorem{proposition}{Proposition}
\newtheorem{remark}{Remark}

\calclayout
\begin{document}
\date{\today}
\title{Lyapunov functions for Morse-Smale synchronisation diffeomorphisms}
\author{Jorge Buescu$^1$ and Henrique M. Oliveira$^{2*}$}
\thanks{$^1$jsbuescu@ciencias.ulisboa.pt; 
ORCID: 0000-0001-5444-5089; Departamento de Matemática, Faculdade de Ciências, and
CEMS.UL - Center for Mathematical Studies, ULisboa
FCT,  UID/04561/2025, 
 Universidade de Lisboa, Campo Grande, 1749-006 Lisbon, Portugal.}
\thanks{$^2$henrique.m.oliveira@tecnico.ulisboa.pt; ORCID: 0000-0002-3346-4915;  
Departamento de Matem\'atica, Instituto Superior T\'ecnico, and  Centro de An\'alise Matem\'atica, Geometria e Sistemas Din\^amicos,
FCT, UID/04459/2025,  Universidade de Lisboa, Av. Rovisco Pais 1, 1049-001 Lisbon, Portugal;\\$^*$ Corresponding author;  }


\begin{abstract}
This paper investigates the dynamical system governing the phase differences between three identical oscillators arranged symmetrically and coupled by  burst interactions. By constructing a discrete Lyapunov function, we prove the existence of two asymptotically stable fixed points on the $2$-torus $\mathbb{T}^2$, which correspond to Huygens synchronisation of three clocks. The locked states have phase differences of $\left( \frac{2\pi}{3}, \frac{4\pi}{3} \right)$ and $\left( \frac{4\pi}{3}, \frac{2\pi}{3} \right)$. Each fixed point possesses an open basin of attraction. The closure of the union of the basins of attraction of the two asymptotically stable attractors is the torus $\mathbb{T}^2$, implying that Huygens synchronisation occurs generically 
and with full Lebesgue measure with respect to initial conditions.

The Morse-Smale nature of the system ensures structural stability, enabling our results to extend to a family of topologically conjugate diffeomorphisms. A common Lyapunov function shared across this family 
%serves as a unified tool for stability analysis, demonstrating 
shows that the above mentioned features of the dynamics persist under small perturbations: oscillators with slightly different natural frequencies still achieve Huygens synchronisation in one of two asymptotically stable states generically and with probability one.

The analogous situation occurs for nearest-neighbour interaction of three slightly different oscillators on a line. In this case, there is a unique open-basin attractor for near-phase opposition synchronisation, which results from a perturbation of the sink at
 $\left( \pi, \pi \right)$ of the original system.
\end{abstract}
\maketitle

\section{Introduction and preparatory results}
\label{sec_intro}

\subsection{Motivation and organisation of this article}

The phenomenon of Huygens synchronisation, first observed by Christiaan Huygens in 1665 \cite{Huy}, describes the spontaneous synchronisation of two  pendulum clocks by the action of weak mechanical coupling. This remarkable behaviour has been widely studied as a paradigmatic example of self-organisation in dynamical systems. The synchronisation of two or three clocks illustrates the interplay between nonlinear coupling and collective dynamics, which serves as the basis for understanding more complex synchronisation phenomena. The study of chaotic synchronisation, where interacting chaotic systems in some sense align their dynamics, has expanded this framework to encompass complex behaviour.
Significant contributions to this field include studies on general principles of synchronisation \cite{Pit}, the mathematical modelling of coupled oscillators \cite{strogatz2004}, and the emergence of collective behaviour in coupled systems with symmetry \cite{Stewart2003SymmetryGA}. The works of Ashwin, Buescu, and Stewart \cite{ASHWIN1994126,ashwin1996attractor,buescu2012exotic}
further explore the subtle bifurcation phenomena related to chaotic synchronisation.

This paper addresses the stability of attractors in diffeomorphisms arising from the synchronisation problem of three clocks arranged in a ring, modeled with all-to-all interactions, and in a line, modeled with nearest-neighbour interactions.

Our main result establishes that, in the system of three identical oscillators interacting via a perturbative all-to-all coupling, there exist two asymptotically stable synchronised states, each possessing a strict Lyapunov function. Moreover, attraction to one of the synchronised states occurs with probability one in phase space. Furthermore, this system  is Morse-Smale and therefore structurally stable. The Lyapunov function is then lifted, via topological conjugacy, to a Lyapunov function for the perturbed diffeomorphism, which represents the physical scenario of slightly different clocks arranged in a ring. Consequently, the perturbed, asymmetric system also exhibits two asymptotically stable locked synchronised states, each  with an open basin of attraction. Correspondingly, the union of these basins is open, dense, and of full Lebesgue measure, implying that synchronisation is generic and occurs with probability one.

The case of identical oscillators arranged in a line with nearest-neighbour interactions is discussed in \cite{BAO2025}, where a discrete Lyapunov function is constructed for the unique sink \( (\pi, \pi ) \). The corresponding synchronisation diffeomorphism is again Morse-Smale. This result establishes that the unique synchronised state corresponding to phase opposition in successive clocks is a robust final state, even when the oscillators are slightly different. The main results for this setting are presented in the conclusions, as most details regarding the construction of the Lyapunov function were provided in a previous paper \cite{BAO2025}.

In the first section, we present the general theory on Morse-Smale diffeomorphisms and Lyapunov functions relevant for our purposes. In Section \ref{sec_DS}, we introduce the dynamical system for the ring model under study along with its Lyapunov function, and state the main results.

To improve readability and clarity we  present the proofs regarding the negativity of the orbital derivative of the Lyapunov function separately in Section \ref{sec_construction}, since they involve intricate and extensive computations. This allows us to spotlight the main results in the earlier sections, allowing the paper to be read without delving too deeply into  technical details. 

The conclusions of this article are presented in Section \ref{sec:_conclusion}, where we summarise our findings.

These results pertain to two different synchronisation diffeomorphisms, corresponding to the geometric settings of a ring with all-to-all interactions and a line with nearest-neighbour interactions. They extend previous studies on identical oscillators. They provide a robust foundation for establishing that phase locking in real physical systems, where perfectly identical clocks do not exist, is structurally stable, provided the natural frequencies of the oscillators are sufficiently close. This conclusion holds for both interaction models considered.

\subsection{Discrete Lyapunov functions}
\label{subsec_discLyap}

Stability analysis of equilibrium points is a fundamental question in dynamical systems and control theory. Lyapunov's approach \cite{lyapunov1892} to stability questions, sometimes called ``Lyapunov’s second method''  \cite{
bertram1960,giesl2015,
Gu1983,hirsch2012differential,lasalle1976stability,sassano2013dynamic}, offers a method to determine the stability or asymptotic stability of equilibrium points without directly calculating the solution of the governing differential equations. 

We begin by briefly reviewing the Lyapunov Stability Theorem for continuous-time systems.

\begin{definition}\label{def:LyapDiff}
Let $f$ be a $C^1$ vector field in $\mathbb{R}^n$. Consider the dynamical system described by the differential equation  
\begin{equation}\label{eq:equadiff}
\frac{\mathrm{d}x}{\mathrm{d}t} = f(x)
\end{equation}
and denote the corresponding flow by $\Phi_t(x)$. Let $x_0$ be a zero of the vector field $f$,
or equivalently a fixed point of $\Phi_t$. Then:
\begin{enumerate} 
\item $x_0$ is {\em Lyapunov stable} if for any neighbourhood $U(x_0)$ there exists a neighbourhood $V(x_0)$ such that 
for all $x \in V(x_0)$ we have $\Phi_t(x) \subset U(x_0)$ for all $t \geq 0$;
\item $x_0$ is {\em asymptotically stable} if it is Lyapunov stable and, in addition, there exists a neighbourhood $W(x_0)$ 
such that for all $x \in W(x_0)$ we have $\lim_{t \to +\infty} \Phi_t(x) = x_0$.
\end{enumerate}
\end{definition}

\begin{theorem}[Lyapunov Stability Theorem] 
\label{Lyap_ODEs}
If there exists a scalar function $V(x)$, called Lyapunov function, for the  dynamical system and equilibrium point mentioned in Definition  \ref{def:LyapDiff}, satisfying the following conditions:
\begin{enumerate}
    \item $V(x_0) = 0$ and $V(x) > 0$ for all $x \neq x_0$,
    \item $\frac{dV \left(x\left(t\right)\right)}{dt} \leq 0$ for all $x \neq x_0$,
\end{enumerate}	
then the equilibrium point $x_0$ is Lyapunov stable. If, additionally, $\frac{dV \left(x\left(t\right)\right)}{dt} < 0$ for all $x \neq x_0$, then $ x_0$ is asymptotically stable.
\end{theorem}

Thus, a Lyapunov function $V(x)$ is a positive definite function which is non-increasing along orbits of the system \eqref{eq:equadiff}, and this implies stability of the equilibrium point. If $V(x)$ is strictly decreasing, then the equilibrium point is asymptotically stable.

The construction of Lyapunov functions is, in general, difficult. While there are some analytical techniques for constructing Lyapunov functions, such as energy-based methods for mechanical or electrical systems, these are often limited to specific classes of systems. For general nonlinear systems the analytical construction of Lyapunov functions is an extremely difficult task. The key requirement of positive definiteness and (semi-)negativeness of the orbital derivative $dV \left(x\left(t\right)\right)/dt$ in a neighbourhood of the equilibrium point can be a very hard property to prove for systems with a high number of variables. Numerical methods for constructing Lyapunov functions, such as sum-of-squares (SOS) optimization, can suffer from numerical instabilities, especially for systems with a large number of variables or high-degree polynomials \cite{giesl2015}. In fact, even when their existence can be proven through converse theorems \cite{conley1978,sassano2013dynamic} such proofs are typically non-constructive. 

In this paper 
we deal with discrete dynamical systems defined by iteration of diffeomorphisms, so it will be essential to have a 
discrete-time counterpart of Lyapunov’s stability. In this context, contrary to what happens
for ODEs, discrete-time Lyapunov functions are only required to be continuous, see e.g. La Salle \cite{lasalle1976stability}.
This means that discrete Lyapunov functions belong to the category of topological dynamics. 

We  state the results regarding discrete Lyapunov functions in this setting, even though we will later be interested 
specifically in the case of diffeomorphisms of $\mathbb{R}^2$ and $\mathbb{T}^2$. 
%Finished until here 1

\begin{definition}
\label{TDS}
Let $X$ be a topological space and $f: X \to X$ be a continuous map. We will refer to the pair $(X, f)$
as a {\em topological dynamical system}. 
\end{definition}

Observe that this should really be called a {\em semi}dynamical system, but since this point is not relevant for what follows we drop the distinction.

\begin{definition}[Discrete orbital derivative]
\label{def_orb_deriv}
Let $(X,f)$ be a topological dynamical system and $V: X \to \mathbb{R}$  be a continuous function. 
We define the {\em discrete orbital derivative of $f$\/} as the function
\[ \dot{V}(x) = V(f(x)) - V(x). \]
\end{definition}

If $(x_n)_{n \in \mathbb{N}}$ is an orbit of the semidynamical system defined by $f$ (i.e. $x_{n+1}= f(x_n)$), then $\dot{V}(x_n) =  V(x_{n+1}) - V(x_{n})$, and therefore 
$\dot{V}(x) \leq 0$ means that $V$ is nonincreasing along orbits of $f$.

\begin{definition}[Discrete Lyapunov function]
\label{def_discrete_Lyap}
Let $(X,f)$ be a topological dynamical system and $V: X \to \mathbb{R}$  be a continuous function.
Suppose $S \subset X$. We say that $V$ is a {\em Lyapunov function for $f$ on $S$\/} if %$V$ is continuous and 
$\dot{V}(x) \leq 0$ for all $x \in S$.
\end{definition}

The following is the discrete version of Lyapunov's Stability Theorem, whose statement directly follows from 
 \cite{lasalle1976stability}. In 
the statement below, Lyapunov stability and asymptotic stability are
simply the discrete counterparts of the corresponding concepts in Definition \ref{def:LyapDiff}.

\begin{theorem}[Discrete Lyapunov Stability Theorem]
\label{Lyapunov for maps}
Let $(X,f)$ be a topological dynamical system.
Let $H$ be compact and let $S$ be an open set containing $H$. Suppose that $V(x)$ is a function such that 
\begin{enumerate}
\item $V(x) \leq 0$ for $x \in H$ and $V(x) > 0$ for $x \in S \setminus H$, and
\item $V$ is a Lyapunov function for $f$ on $S$.
\end{enumerate}
Then $H$ is Lyapunov stable. If, in addition, $\dot{V}(x) <0$ on $S \setminus H$, then $H$ is asymptotically stable.
\end{theorem}

To state our next result we need to recall the notion  of topological conjugacy of dynamical systems.

\begin{definition}
\label{def_TC}
 We say that two topological dynamical systems $(X_1, f_1)$ and $(X_2, f_2)$ are {\em topologically 
conjugate\/} if there exists a homeomorphism $h: X_1 \to X_2$ such that 
\begin{equation}
\label{eqn_TC}
h \circ f_1 = f_2 \circ h.
\end{equation}
A homeomorphism $h$ satisfying \eqref{eqn_TC} is called a topological conjugacy between $f_1$ and $f_2$.

\end{definition}
%When two dynamical systems $(X_1, f_1)$ and $(X_2, f_2)$ are topologically conjugate, a homeomorphism 
%$h$ performing the 
%conjugation, i.e. satisfying \eqref{eqn_TC}, is called a topological conjugacy between $f_1$ and $f_2$.

\begin{theorem}
\label{thm_conjugacy}
Let $(X_1, f_1)$ and $(X_2, f_2)$ be topologically conjugate. Suppose that the system $(X_1, f_1)$ admits
a Lyapunov function $V_1$, and let $S_1, \, H_1$ be the associated open and compact sets 
in Theorem~\ref{Lyapunov for maps}. Then
\begin{equation}
V_2 = V_1 \circ h^{-1}
\label{eq_Lyap_TC}
\end{equation}
is a Lyapunov function for $f_2$ associated to the sets $S_2= h(S_1)$  and $H_ 2 = h(H_1)$.
\end{theorem}

\proof 
Suppose $x_2 \in X_2$ and let $x_1 = h^{-1}(x_2)$. Then 
\begin{align*}
\dot{V_2}(x_2) &= V_1(h^{-1}(f_2(x_2))) - V_1(h^{-1}(x_2))\\
&=V_1(f_1(h^{-1}(x_2))) - V_1(h^{-1}(x_2)) \\
&= V_1(f_1(x_1)) - V_1(x_1).
\end{align*}
so $V_2$ is a Lyapunov function for the topological dynamical system $(X_2, f_2)$. Moreover, consideration 
of the commutative diagram
\begin{center}
\begin{tikzcd}
    X_1 \arrow[r, "f_1"] \arrow[d, "h"'] & X_1 \arrow[d, "h"] \\
    X_2 \arrow[r, "f_2"'] & X_2
\end{tikzcd}
\end{center}
shows that, if the sets $S_1$ and $H_1$ have the properties stated 
in Theorem~\ref{Lyapunov for maps} for 
$(X_1, f_1)$, then
 the corresponding sets $S_2= h(S_1)$  and $H_ 2 = h(H_1)$ have those
 properties for $(X_2, f_2)$.
 \qed

\vspace{2mm}
Observe that in any topological space $X$ a finite set is always compact, so 
Theorems~\ref{Lyapunov for maps} 
and \ref{thm_conjugacy} are, in particular, immediately applicable to the study of the stability of fixed points. 
Indeed, if  $x^*$ is a fixed point of $f$, then $H= \{ x^* \}$ is compact and the Lyapunov method applies.

In this paper we deal specifically with diffeomorphisms of $\mathbb{T}^2$, and the compact
positively invariant sets $H$ we consider consist in fixed points.
For a recent application of Lyapunov functions in discrete maps see \cite{Baigent2023}. 

\subsection{Structural stability and Morse-Smale systems}
\label{subsec_SS+MS}

We now recall two basic definitions which will prove crucial in the rest of the paper. This section will be
formulated in the context of diffeomorphisms, so we now switch to the category of differentiable dynamics.
Throughout this section $M$ will denote a smooth, e.g. $C^\infty$, manifold, $C^1(M,M)$ the set of $C^1$ 
self-maps of $M$ and ${\rm Diff}(M) \subset C^1(M,M)$ the group of all $C^1$ diffeomorphisms of $M$
equipped with the $C^1$ norm. Note, in particular, that unlike  topological (semi)dynamical  systems,
a diffeomorphism $f$ is a dynamical system, i.e. $f^n$ is a diffeomorphism for all $n \in \mathbb{Z}$.

The definition of structural stability,  due to Andronov and Pontrjagin in the late 1930's and published originally in Russian \cite{Andronov1937} (for an English version published in 1971 by Andronov collaborators see \cite{andronovtheory}), 
encapsulates the idea of robustness of a system's qualitative behaviour under small perturbations, arising for example 
from small changes in the parameters. This concept was further developed by Peixoto in now classical works \cite{Peixoto1959,Peixoto1962}.

\begin{definition}
\label{def_SS}
Let $M$ be a smooth manifold and $f \in {\rm Diff}(M)$. Then $f$ is said to be {\em structurally stable}
if there exists a neighbourhood $N(f)$ of $f$ in ${\rm Diff}(M)$ such that every $g \in {\rm Diff}(M)$ is
topologically conjugate to $f$.
\end{definition}

Observe that the definition of structural stability uses the $C^1$ norm in ${\rm Diff}(M)$ but the $C^0$ norm
in the topological conjugacy.

We now address the definition of a class of diffeomorphisms which will play a pivotal role below:
the Morse-Smale \cite{PALIS1969,palis2000structural,Smale1960,smale1967}  dynamical systems. 
For $f \in {\rm Diff}(M)$ the orbit of $x$ is the set $O(x)= \{f^n(x)\}_{ n \in \mathbb{Z}}$. We denote by $\Omega(f)$ the set of nonwandering points of $f$ (see \cite{Gu1983,Hale1984,Palis1982} for  definitions, extensions and  general background).

\begin{definition}
\label{def_MS}
Let $M$ be a compact smooth manifold and  $f \in {\rm Diff}(M)$. Then $f$ is said to be a Morse-Smale system if:
\begin{enumerate}
\item $\Omega(f)$ consists of a finite number of fixed points or periodic orbits, all of them hyperbolic;
\item for all $x \in M$ the limit sets of the orbit $O(x)$ are either a fixed point or a periodic orbit.
\item the stable and unstable manifolds of all fixed points and periodic orbits intersect transversely.
\end{enumerate}
\end{definition}

Morse-Smale systems have a simultaneously rich but simple structure, allowing for 
a thorough characterization. A basic property of Morse-Smale systems is that of
structural stability \cite{palis1970structural}.

\begin{theorem}[Palis-Smale 1970]
\label{thm_MS_are_SS}
Let $M$ be a compact smooth manifold and  $f \in {\rm Diff}(M)$ be Morse-Smale. Then $f$ is structurally stable.
\end{theorem}

The following result is now a consequence of Theorems~\ref{thm_conjugacy} and \ref{thm_MS_are_SS}.

\begin{theorem}
\label{Lyap_for_MS}
Let $f \in {\rm Diff}(M)$ be Morse-Smale. Suppose $x_0$ is an asymptotically stable 
fixed point of $f$ admitting a Lyapunov function $V(x)$. 
Then, for all sufficiently $C^1$-close $\tilde{f} \in {\rm Diff}(M)$, 
$\tilde{x_0}=h(x_0)$ is an asymptotically stable fixed point 
admitting the Lyapunov function $\tilde{V} = V \circ h^{-1}$, where $h$ is a conjugacy between $f$ and $\tilde{f}$.
\end{theorem}

\proof Since $f$ is Morse-Smale, it is structurally stable by Theorem~\ref{thm_MS_are_SS}. It follows that there exists a neighbourhood $N(f)$ in 
${\rm Diff}(M)$ such that, for all $\tilde{f} \in N(f)$, $\tilde{f}$ is topologically conjugate to $f$. Let $h$ be a topological conjugacy from $f$ to $\tilde{f}$ that is a homeomorphism satisfying
\[  h \circ f = \tilde{f} \circ h. \]
Then $\tilde{x_0}=h(x_0)$ is an asymptotically stable fixed point for $\tilde{f}$.
It now follows from Theorem~\ref{thm_conjugacy} that $\tilde{V} = V \circ h^{-1}$ is a Lyapunov function for $\tilde{f}$, as asserted. 
\qed

\begin{remark}
\label{rem_periodic_orbit}
We note that this result extends naturally from fixed points to periodic orbits, since
if $(x_k)_{k \in \mathbb{Z}}$ is an asymptotically stable periodic orbit of period $n$ for a Morse-Smale system $f$, then it is an asymptotically stable fixed point for $f^n$, which is also a Morse-Smale system. This shows that, in fact, Theorem~\ref{Lyap_for_MS}
applies to any asymptotically stable attractor of a Morse-Smale system.
\end{remark}

\section{The synchronisation diffeomorphism for three equidistant clocks}\label{sec_DS}

\subsection{Identical clocks}

In a series of recent papers \cite{BAO2024b,EH2}, 
the authors
investigated the synchronisation of three plane oscillators with an
asymptotically stable limit cycle under the mechanism of Huygens
synchronisation of the second kind, that is, where the interaction is performed not via
momentum transfer but by a perturbative mechanism. The model incorporated the
Andronov  pendulum clock \cite{And} as used in \cite{OlMe}, but the method 
applies as well to other types of oscillators with coupling given by the discrete
Adler equation \cite{adler1946study,Pit}. The theory only depends on systems
having limit cycles and small interactions between oscillators once per cycle,
ensuring applicability irrespective of the specific details of the oscillator
models. We refer to these oscillators as clocks, since we assume isochronism
of each oscillator when isolated from perturbations.

\begin{figure}[tbh]
\centering
\includegraphics[height=3in,width=3.513in]{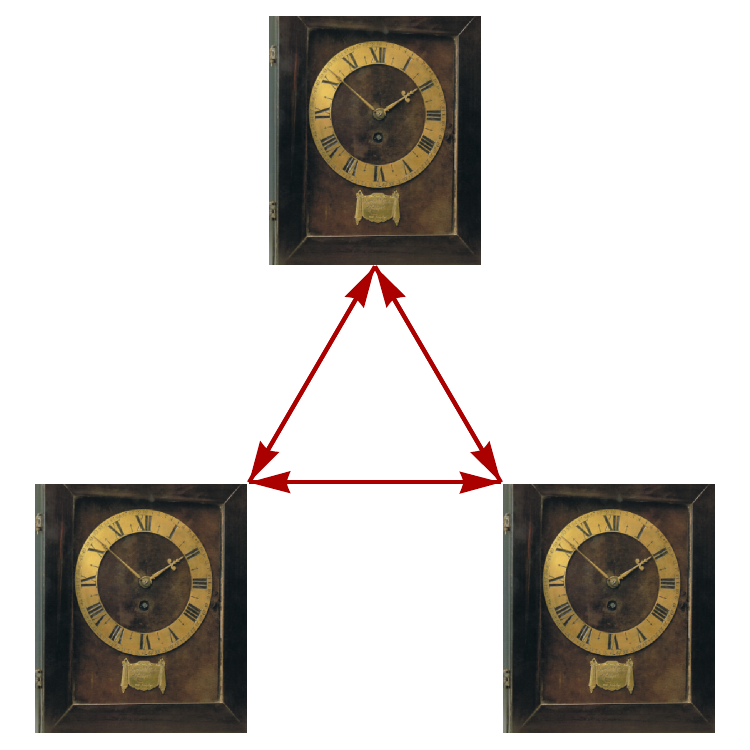}\caption{System of three symmetrically coupled clocks.}
\label{fig:_01}
\end{figure}

Consider the case of symmetric interaction between all three clocks, as  shown in Fig.~\ref{fig:_01}.
We refer to the three clocks by $A$, $B$ and $C$. 
%We start by fixing one of the clocks, say $A$, and denote the two others by $B$ and $C$. 
We denote by $x$ and $y$ the phase differences of clocks $B$ and $C$ relative to $A$.
In \cite{EH2} it is shown that the dynamics is described by the discrete system in $\mathbb{R}^2$
\begin{align}\label{eq:_4}
\begin{bmatrix}
x_{n+1}\\
y_{n+1}%
\end{bmatrix}
=G
\begin{bmatrix}
x_{n}\\
y_{n}%
\end{bmatrix}
=%
\begin{bmatrix}
x_{n}+2a\sin x_{n}+a\sin y_{n}+a\sin(x_{n}-y_{n})\\
y_{n}+a\sin x_{n}+2a\sin y_{n}+a\sin(y_{n}-x_{n})
\end{bmatrix}.
\end{align}
We denote 
the components of the vector field $G$ by 
\begin{equation}
\label{eq_G}
G(x,y)=\left( g_1\left(x,y \right) ,g_2\left(x,y\right) \right).
\end{equation}

%We note that
%\begin{equation}
%\label{eq_gcommute}
% g_2\left( x,y\right)=g_1\left( y,x\right).
%\end{equation}

\begin{remark}
\label{rem_param_region}
It is easily shown that the vector field $G$ in  \eqref{eq:_4} is a diffeomorphism for $0<a<\frac{1}{3}$.
%\begin{equation} 
%0<a<\frac{1}{3}.\label{eq:acond}%
%\end{equation}
Throughout the rest of the paper we will work within this parameter region for $a$. 
\end{remark}

In \cite{EH2}, the authors fully characterize 
the dynamics of the system \eqref{eq:_4}, which we now proceed to describe.
First of all, it is immediate to see that the dynamics is periodic with period $2\pi$ in both variables.
Consider as fundamental domain  the square
\[
D=[0,2\pi]\times\lbrack0,2\pi].
\]
\noindent There exist 11 fixed points in  $D$:
\begin{itemize}
\item[(i)] hyperbolic unstable nodes (sources) at $(0,0)$,
$(0,2\pi)$, $(2\pi,0)$, and $(2\pi,2\pi)$; 
\item[(ii)] hyperbolic saddle points at $(\pi,0)$,
$(0,\pi)$, $(2\pi,\pi)$, $(\pi,2\pi)$ and $(\pi,\pi)$; 
\item[(iii)] hyperbolic asymptotically stable nodes
(sinks) at $(\frac{2\pi}{3},\frac{4\pi}{3})$ and $(\frac{4\pi}{3},\frac{2\pi
}{3})$.
\end{itemize}
%\end{remark}
Periodicity of the system then implies that the phase space $\mathbb{R}^2$ is tiled by translations of the square $D$,
so the dynamics may be considered on quotient space, the 2-torus $\mathbb{T}^{2} = \mathbb{R}^2/\left(2\pi\mathbb{Z}\right)^2$. 

In \cite{EH2}, the dynamical system was analysed in the square \( D \subset \mathbb{R}^2 \). However, in this article, it will be most convenient to 
consider the dynamical system on the torus $\mathbb{T}^{2}$. 
By a slight abuse of notation, but without risk of confusion, we also use $G$ 
to denote the induced dynamical system
 on the torus, that is, the system $\left(\mathbb{T}^2,G   \right)$ corresponding to the iteration
\begin{equation}
\label{eq_torus}
\left\{
\begin{aligned}
x_{n+1} \equiv &  \ g_1\left(x_{n},y_{n}   \right) \mod 2 \pi \\
y_{n+1} \equiv &  \ g_2\left(x_{n},y_{n}   \right) \mod 2 \pi.
\end{aligned}
\right.
\end{equation}

Taking into account the identifications induced in $D$ by the quotient $\mathbb{T}^{2} = \mathbb{R}^2/\left(2\pi\mathbb{Z}\right)^2$, 
the 11 fixed points in $D$ correspond to the following 6 fixed points  on $\mathbb{T}^{2}$:

\begin{itemize}
\item[(i)] one hyperbolic unstable node (source) at $(0,0)$; 
\item[(ii)] three hyperbolic saddles at $(0,\pi)$, $(\pi,0)$ and $(\pi,\pi)$;
\item[(iii)] two hyperbolic stable nodes (sinks) at $(\frac{2\pi}{3},\frac{4\pi}{3})$ and $(\frac{4\pi}{3},\frac{2\pi
}{3})$. 
\end{itemize}

\noindent Moreover, there are no saddle-saddle heteroclinic connections, so all stable and unstable manifolds 
of the fixed points intersect transversally.

Since all fixed points are hyperbolic in the parameter window under consideration $0 < a < \frac{1}{3}$ (recall Remark~\ref{rem_param_region}), we shall henceforth refer to an unstable node as a source and to a
stable node as a sink. Saddles will be simply referred to as saddles. 

\begin{figure}[ptb]
\begin{center}
\includegraphics[
height=4.5in,
width=4.5in
]{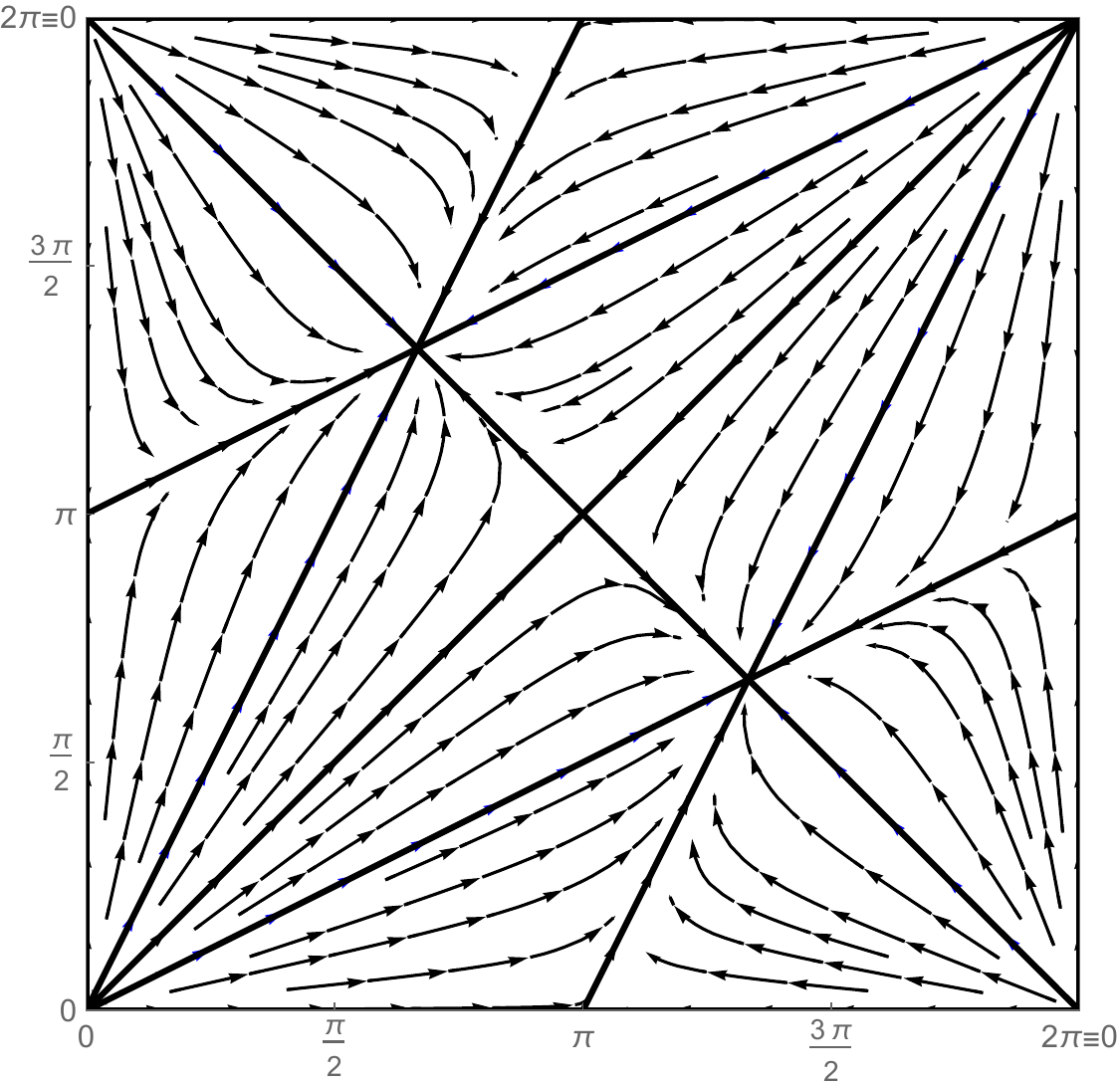}
\end{center}
\caption{A planar representation of the torus $\mathbb{T}^2$ using coordinates $x$ and $y$. In light gray, we display the streamlines of the dynamical system. The opposite edges correspond to the same vertical and horizontal sections on the torus via the canonical identification map.
}
\label{fig:_02}
\end{figure}

\noindent We summarize the discussion of \cite{EH2} in the next Remark and in Fig.~\ref{fig:_02}. 

\begin{remark}\label{connections}
There are no saddle-saddle connections or homoclinic connections, so saddles connect only to sinks or sources.

There exist 18 straight-line segment invariant sets in $\mathbb{T}^2$. 
%The heteroclinic connections between saddle points and other fixed points are the most important invariant sets in our study, due to the Morse-Smale property. 
Naturally, all orbits connecting sources and sinks are heteroclinic, but we are interested only in the straight-line segments, which will be essential for the constructions in the rest of the paper.

The 
18 invariant line segments on the torus are depicted in Fig.~\ref{fig:_02}, where 
identification of the edges of the square must be taken into account. We enumerate them as follows.

\begin{enumerate}
\item the source $(0,0)$ and the saddle $(0,\pi)$ are connected by two
heteroclinics.

\item the source $(0,0)$ and the saddle $(\pi,0)$ are connected by two
heteroclinics.

\item the  source $(0,0)$ and the saddle $(\pi,\pi)$ are connected by two heteroclinics.

\item the saddle $(0,\pi)$ and the sink $(\frac{2\pi}{3},\frac{4\pi
}{3})$ are connected by a heteroclinic.

\item the saddle $(\pi,0)$ and the sink $(\frac{2\pi}{3}%
,\frac{4\pi}{3})$ are connected by a
heteroclinic.

\item the saddle $(\pi,\pi)$ and sink $(\frac{2\pi}{3}%
,\frac{4\pi}{3})$ are connected by a
heteroclinic.

\item the source $(0,0)$ and the sink $(\frac{2\pi}{3}%
,\frac{4\pi}{3})$ are connected by three heteroclinics which are straight line segments.

\item the saddle $(0,\pi)$ and the sink $(\frac{4\pi}{3},\frac{2\pi
}{3})$ are connected by a heteroclinic.

\item the saddle $(\pi,0)$ and the sink $(\frac{4\pi}{3}%
,\frac{2\pi}{3})$ are connected by a
heteroclinic.

\item the saddle $(\pi,\pi)$ and sink $(\frac{4\pi}{3}%
,\frac{2\pi}{3})$ are connected by a
heteroclinic.

\item the source $(0,0)$ and the sink $(\frac{4\pi}{3}%
,\frac{2\pi}{3})$ are connected by three heteroclinics which are straight 
line segments.

\end{enumerate}

%Observe that there do not exist saddle-saddle connections nor homoclic connections.
\end{remark}

The next proposition 
 is now an immediate consequence of  the previous discussion 
of the dynamics on $\mathbb{T}^2$ coupled with 
 Remark~\ref{connections}. 

\begin{proposition}\label{prop:_MS}
The  diffeomorphism $G: \mathbb{T}^2 \to \mathbb{T}^2$  is Morse-Smale. The nonwandering set $\Omega(G)$
consists on the six hyperbolic fixed points: three saddles, two sinks and one source. 
\end{proposition}

In view 
of the phase space dynamics just described and of the physical nature of the problem (synchronisation of three
identical clocks) we may state the following result.

\begin{corollary}
\label{cor_prob_1}
Almost all initial conditions, in the sense of Lebesgue measure,  on $\mathbb{T}^2$ approach one of the two synchronised states $(\frac{2\pi}{3} ,\frac{4\pi}{3})$ or $(\frac{4\pi}{3} ,\frac{2\pi}{3})$, and do so exponentially fast.
\end{corollary}

\begin{proof} The basins of the two asymptotically stable attractors $(\frac{2\pi}{3} ,\frac{4\pi}{3})$ or $(\frac{4\pi}{3} ,\frac{2\pi}{3})$ are open and the only points not in one of the basins are the source to saddle connections, which have Lebesgue measure zero. 
Thus the union of the basins has full measure. Exponential rates of attraction are a consequence of hyperbolicity.
\end{proof}

Since the torus is compact and has finite Lebesgue measure, by normalisation we may restate this result in terms of the corresponding probability measure, leading to the conclusion that, with probability $1$,  every initial condition on the torus approaches one of the two synchronised states. 

We note that from the topological point of view the union of the basins is, of course, an open and dense set on $\mathbb{T}^2$, therefore, synchronisation is also generic.

\subsection{Non-identical clocks}

In \cite{EH2} and \cite{BAO2024b}, the authors consider identical clocks. However, since perfectly identical oscillators do not exist in nature, we aim to describe the dynamics of a modified system that allows for oscillators with small differences in angular frequencies\footnote{That is, $\omega=2\pi/T$, where $T$ is the natural period of the clock.} between clocks $B$ and $C$ relative to clock $A$.
%, with
%respective magnitudes $\delta_{1}$ and $\delta_{2}$.

To 
study this problem we consider a new, perturbed diffeomorphism $\widetilde{G}$ of the torus $\mathbb{T}^2$: 
\begin{equation}
\begin{bmatrix}
x_{n+1}\\
y_{n+1}
\end{bmatrix}
=\widetilde{G}
\begin{bmatrix}
x_{n}\\
y_{n}
\end{bmatrix},
\label{eq:nonidentical}
\end{equation}
in which we add a small $C^1(\mathbb{T}^2)$ perturbation to the original vector field $G$ on the torus:
\begin{equation}
\widetilde{G}=
\begin{bmatrix}
x+2a\sin x+a\sin y+a\sin(x-y)+\delta_{1} \zeta_1(x,y)\\
y+a\sin x+2a\sin y+a\sin(y-x)+\delta_{2} \zeta_2(x,y)
\end{bmatrix}, 
\end{equation}
where $\delta_{1}$ and $\delta_{2}$ are small perturbation parameters.  
The diffeomorphism $\widetilde{G}$ can be used to model general perturbations of the phase differences of the three oscillators, including (small) periodic external forcing.

\begin{remark}\label{rem:close}
In the particular case of near-identical clocks with close natural frequencies, the perturbation functions are much simplified
\[
\left( \zeta_1(x,y),\zeta_2(x,y)\right) \equiv (1,1).
\]
\end{remark}

In a similar fashion to the fully symmetric case \eqref{eq_G}, we denote the components 
of the vector field $\widetilde{G}$ by 
\begin{equation}
\label{eq_tildeG}
\widetilde{G}(x,y)=\left( \tilde{g}_1\left(x,y \right) ,\tilde{g}_2\left(x,y\right) \right).
\end{equation}

The dynamical system on the torus $\left(\mathbb{T}^2,\tilde{G} \right)$ is then written 
\begin{equation}
\label{eq_torus_perturbed}
\left\{
\begin{aligned}
x_{n+1} & \equiv \tilde{g}_1\left( x_{n},y_{n} \right) \mod 2\pi, \\
y_{n+1} & \equiv \tilde{g}_2\left(x_{n},y_{n} \right) \mod 2\pi,
\end{aligned}
\right.
\end{equation}
that is
\begin{equation}
\label{eq_torus_perturbed2}
\left\{
\begin{aligned}
x_{n+1} &\equiv g_1\left(x_{n},y_{n} \right) + \delta_1 \zeta_1(x_n,y_n) \mod 2\pi, \\
y_{n+1} &\equiv g_2\left(x_{n},y_{n} \right) + \delta_2 \zeta_2(x_n,y_n) \mod 2\pi.
\end{aligned}
\right.
\end{equation}

From \eqref{eq_torus_perturbed2}, it follows that 
$G$ and $\tilde{G}$ are close in the \( C^1\left(\mathbb{T}^2 \right) \) topology whenever $\delta_1, \delta_2$ are small. Indeed, since 
\[
\tilde{G}(x,y) - G(x,y) = (\delta_1 \zeta_1(x,y), \delta_2 \zeta_2(x,y))
\]  
and  
$\mathbb{T}^2$ is compact, 
the norm of $(\delta_1 \zeta_1,\delta_2 \zeta_2)$ is 
\begin{align*}
\| (\delta_1 \zeta_1, \delta_2 \zeta_2) \|_{C^1(\mathbb{T}^2)} &= |\delta_1|A_1+|\delta_2|A_2,
\end{align*}
where the constants $A_j$, $j=1,2$, take the form
\begin{align*}
A_j = \max_{\mathbb{T}^2}{\left|\zeta_j (x,y)\right| } + \max_{\mathbb{T}^2}{\left|\partial_x \zeta_j (x,y)\right| } + \max_{\mathbb{T}^2}{\left|\partial_y \zeta_j (x,y)\right| }.
\end{align*} 

It follows that  
\begin{align}
\label{eq_C1_closea}
\| \tilde{G} - G \|_{C^1} = \| (\delta_1 \zeta_1, \delta_2 \zeta_2) \|_{C^1(\mathbb{T}^2)}
= |\delta_1|A_1+|\delta_2|A_2,
\end{align}
and since \( |\delta_1| \) and \( |\delta_2| \) are as small as necessary, 
the diffeomorphisms are \( C^1 \)-close.

From Remark~\ref{rem:close}, for the case of close constant natural frequencies, the perturbation functions are 
\[
\left( \zeta_1(x,y),\zeta_2(x,y)\right) = (1,1) \text{ for all}  \left( x,y \right) .
\]  
It thus follows that the derivatives of $G$ and $\tilde{G}$ are identical, $A_j= 1$, and therefore
\begin{equation}
\label{eq_C1_close}
\| \tilde{G} - G \|_{C^1} = |\delta_1| + |\delta_2|.
\end{equation}

The following result applies to the general case of a differentiable perturbation of $G$.

\begin{proposition}\label{prop:_conjugated}
For small enough $|\delta_1|, \,  |\delta_2|$, 
the dynamical systems $\left(\mathbb{T}^2,G \right)$ and $\left(\mathbb{T}^2,\tilde{G} \right)$ are topologically conjugate. 
\end{proposition}

\begin{proof}
The dynamical system $\left(\mathbb{T}^2,G \right)$ is Morse-Smale, and therefore structurally stable by Theorem~\ref{thm_MS_are_SS}. On the other hand, it follows from \eqref{eq_C1_close} that for small enough $|\delta_1|, \,  |\delta_2|$, the vector 
fields $\tilde{G}$ and $G$ are arbitrarily $C^1$-close.
\end{proof}

\subsection{Lyapunov function for $G$ and $\tilde{G}$}

We now state the main results of this paper.

\begin{theorem}
\label{thm:_LyapunovG1}
Consider 
the dynamical system $(\mathbb{T}^2,G)$ defined by \eqref{eq_torus}. 
Let $H_1=\left\{ \left(  \frac{2\pi}{3},\frac{4\pi}{3}\right) \right\}$ and $S$ be the open set defined by
\begin{equation}\label{eq:_S}
 S = \{ \left( x,y \right) \in \left] 0, 2 \pi \right[\times \left] 0, 2 \pi \right[ : \,   y> x \}.
 \end{equation}
Then the function
 \begin{equation}
\label{eq_V}
V\left(  x,y\right)= 
\left\vert y-2x\right\vert +\left\vert 2\pi
+x-2y\right\vert \
\end{equation}
\noindent is a strict Lyapunov function for $G$ on $S$ and $H_1$ is asymptotically stable.
\end{theorem}

\begin{theorem}
\label{thm:_LyapunovG2}
Consider 
the dynamical system $(\mathbb{T}^2,G)$ defined by \eqref{eq_torus}.

Let $H_2=\left\{ \left(  \frac{4\pi}{3},\frac{2\pi}{3}\right) \right\}$ and let $R$ be the open set defined by
\begin{equation}\label{eq:_R}
R = \{ \left( x,y \right) \in \left] 0, 2 \pi \right[\times \left] 0, 2 \pi \right[ : \,   y< x \}.
\end{equation} 
Then the function
 \begin{equation}
\label{eq_U}
U\left(  x,y\right)= 
\left\vert x-2y\right\vert +\left\vert 2\pi
+y-2x\right\vert  
\end{equation}
 is 
a strict Lyapunov function for $G$ on $R$ and $H_2$ is asymptotically stable.
\end{theorem}

The proof of Theorems~\ref{thm:_LyapunovG1} and \ref{thm:_LyapunovG2} is 
somewhat involved and is deferred to section \ref{sec_construction}.

Note that $V$ and $U$ are continuous within the basins of attraction $S$ and $R$ of the sinks but are not defined on the entire torus $\mathbb{T}^2$. This is not problematic since, as stated in Definition \ref{def_discrete_Lyap} of Section \ref{sec_intro}, a discrete Lyapunov function is only required to be continuous on an appropriate open set.

%However, as stated in Definition \ref{def_discrete_Lyap} of Section \ref{sec_intro}, a discrete Lyapunov function is only required to be continuous on an appropriate open set.
Moreover, it is clear 
from \eqref{eq_V} and \eqref{eq_U} that 

\begin{enumerate}
\item $V(x,y) \geq 0$ for all $(x,y) \in S$, with equality only at
$\left(  \frac{2\pi}{3},\frac{4\pi} {3}\right)$;
\item $U(x,y) \geq 0$ for all $(x,y) \in R$, with equality only at 
$\left(  \frac{4\pi}{3},\frac{2\pi}{3}\right).$
\end{enumerate}
Using $U$ and $V$ we obtain a continuous function on $D\subset \mathbb{R}^2$, defined as
\[
L(x,y) =
\begin{cases} 
V(x,y), & \text{for } 0 \leq x \leq 2\pi, y \geq x, \\
U(x,y),  & \text{for } 0 \leq x \leq 2\pi, y < x.
\end{cases}
\]
The function $L$ defined above in $D$ is, however, not continuous on the torus $\mathbb{T}^2$ since
\[
U(x,0) \neq V(x,2\pi)\text{, }\forall_{x \neq \pi}
\]
and
\[
V(0,y) \neq U(2\pi,y)\text{, }\forall_{x \neq \pi} \text{.}
\]
So, 
attempting to construct a global Lyapunov function on the whole torus by simply joining the domains of definition of $U$ and $V$ does not succeed since these functions fail to glue together in a continuous manner across the torus. At the end of this paper we suggest a global Liapunov function for the map $G$.

%that  $V\left(
%x,y\right)  >0$ except for the asymptotically stable set $H_1=\left\{ \left(  \frac{2\pi}{3},\frac{4\pi}
%{3}\right) \right\} $, and $U\left(
%x,y\right)  >0$ except for the asymptotically stable set $H_2=\left\{ \left(  \frac{4\pi}{3},\frac{2\pi}
%{3}\right) \right\} $.

\begin{figure}
\centering
\begin{subfigure}[b]{0.5\textwidth}
   \includegraphics[width=1\linewidth]{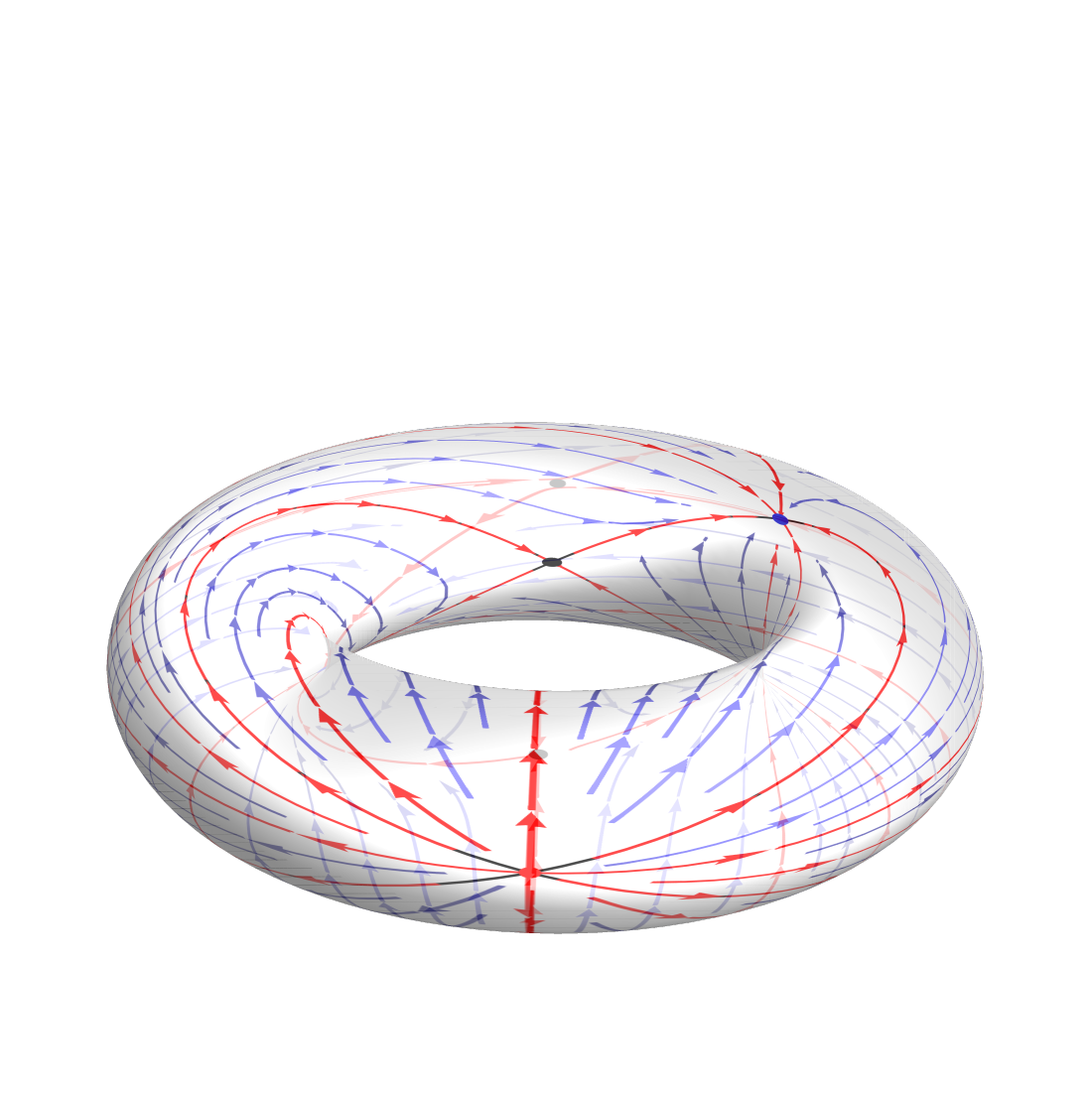}
   \caption{We can see the source at the origin in red, the saddle $\left( \pi,  \pi  \right)$ in black and the asymptotically stable node (sink) $\left( \frac{4 \pi}{3}, \frac{2 \pi}{3}  \right)$ in blue.}
   \label{fig:_05} 
\end{subfigure}
\begin{subfigure}[b]{0.5\textwidth}
   \includegraphics[width=1\linewidth]{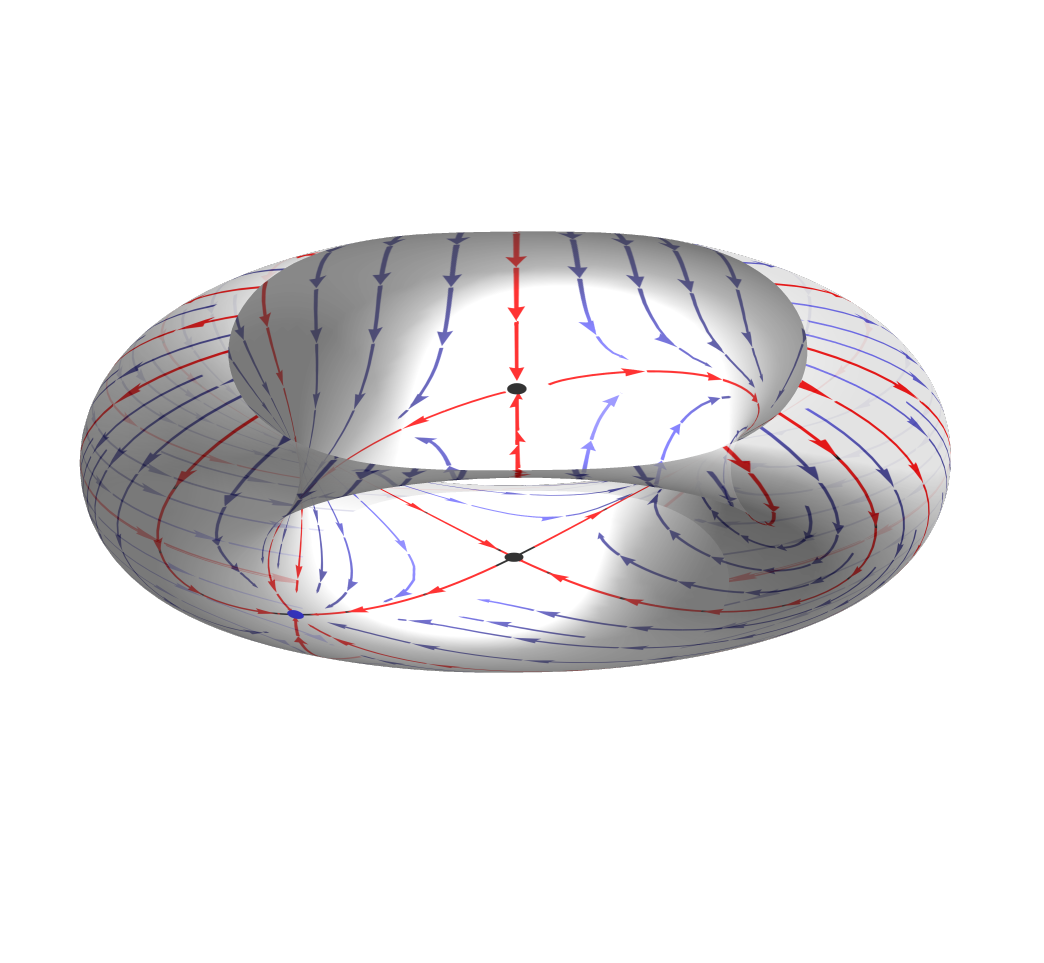}
   \caption{We can see a cut of the torus now with the saddle at 
   $\left( 0,  \pi  \right)$ visible in black in the front, the sink $\left( \frac{ 2\pi}{3}, \frac{4 \pi}{3}  \right)$ in blue in the back and, again, the saddle $\left( \pi,  \pi  \right)$ in black.}
   \label{fig:_06}
\end{subfigure}
\caption{The $\left( \mathbb{T}^2, G \right)$ flow on the torus in two perspectives. }
\end{figure}
In Fig.~\ref{fig:_05}  and \ref{fig:_06} are depicted 
 two views of the flow on the torus, identifying the source, the sinks and the saddles. The bottom view presents a cut so that the saddle on the back side is visible.

\vspace{5mm}

We now focus on the perturbed system $\left( \mathbb{T}^2, \tilde{G} \right)$. 
From Proposition~\ref{prop:_conjugated} it follows that, for small enough $\epsilon = |\delta_1| + |\delta_2|$, there exists a topological conjugacy between 
$G$ and $ \tilde{G}$, namely, a homeomorphism  $h$ 
such that 
\begin{equation}
\label{eq_TC}
 h \circ G = \tilde{G} \circ h.
\end{equation}

The topological conjugacy $h$ maps the source, the sinks and the saddles of $G$ onto the corresponding source, sinks and saddles of $\tilde{G}$. We now show that the sinks for the perturbed system have their own Lyapunov functions corresponding to the ones in Rheorems~ \ref{thm:_LyapunovG1} and \ref{thm:_LyapunovG2} via the conjugacy. More precisely, we have:

\begin{theorem}
\label{thm_Lyap_perturbed}
Let $\left( \mathbb{T}^2, G \right)$ and $\left( \mathbb{T}^2, \tilde{G} \right)$ be as above, and $h$ be a conjugacy as in \eqref{eq_TC}. Then:
\begin{enumerate}
\item $h(\left(  \frac{2\pi}{3},\frac{4\pi} {3}\right))$ is a sink for $\tilde{G}$ with strict Lyapunov function $\tilde{V} = V \circ h^{-1}$ on the open set $h(S)$;
\item $h(\left(  \frac{4\pi}{3},\frac{2\pi} {3}\right))$ is a sink for $\tilde{G}$ with strict Lyapunov function $\tilde{U} = U \circ h^{-1}$ on the open set $h(R)$.
\end{enumerate}
\end{theorem}

\begin{proof} The fact that $h(\left(  \frac{2\pi}{3},\frac{4\pi} {3}\right))$ and 
$h(\left(  \frac{4\pi}{3},\frac{2\pi} {3}\right))$ are sinks is immediate from the 
topological conjugacy. The statements about $\tilde{U}$ and $\tilde{V}$ 
being corresponding Lyapunov functions, respectively, on the open sets $h(S)$ and $h(R)$
is a consequence of Theorem~\ref{thm_conjugacy}.
\end{proof}

In Fig.~\ref{fig:_08} the dynamics on the torus for the perturbed system is depicted.

We may 
now state, for the perturbed system, the conclusion corresponding to  Corollary~\ref{cor_prob_1}.

\begin{corollary}
\label{cor_prob_2}
Almost all initial conditions on $\mathbb{T}^2$ approach one of the two synchronised
states corresponding to the sinks of $\tilde{G}$, and do so exponentially fast.
\end{corollary}

\proof As shown in \ref{cor_prob_1}, in the system $\left( \mathbb{T}^2, G \right)$ every initial condition 
approaches one of the two synchronised states corresponding to a sink except those lying on the source to 
saddle connections, which have zero Lebesgue measure. The topological conjugacy maps these 
connections homeomorphically onto source to saddle connections of the perturbed system. 
These  connections are invariant manifolds of hyperbolic fixed points of the perturbed system, which
is analytic, and so are analytic curves. Therefore they have measure zero on $\mathbb{T}^2$.
\qed

As a consequence of this result, we can state that for the perturbed system, synchronisation to one of the two attracting states occurs with probability 1 with respect to the initial conditions. The genericity of synchronisation is also a consequence of the topological conjugacy between the unperturbed and perturbed systems.

\begin{figure}
\centering
\begin{subfigure}[b]{0.45\textwidth}
   \includegraphics[width=1\linewidth]{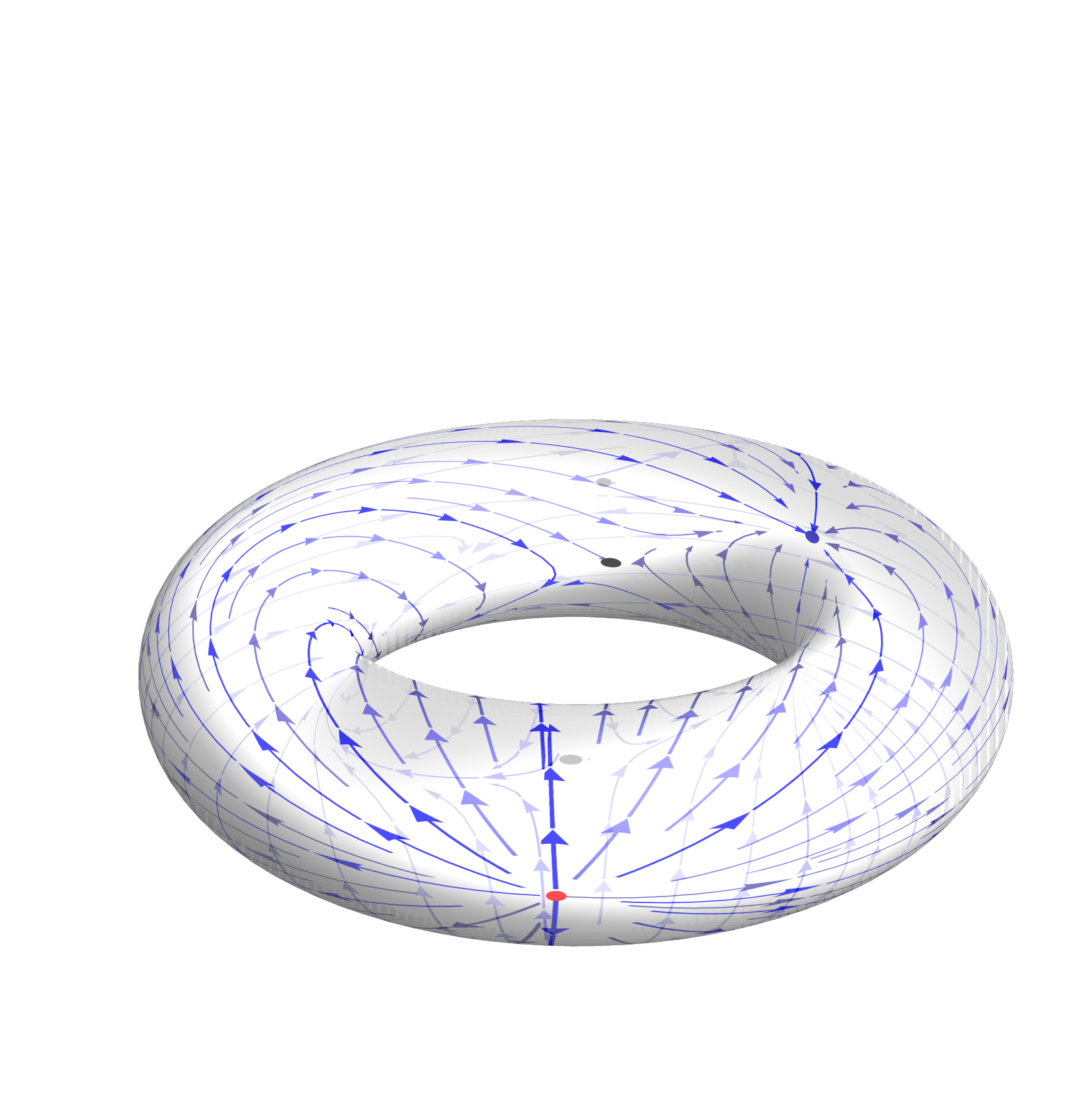}
   \caption{Dynamics for the perturbed dynamical system $\left( \mathbb{T}^2, \tilde{G} \right)$. Since the original dynamical system is structurally stable the phase portrait is only slightly distorted relative to the original. In this case $\delta_1=0.01$ and $\delta_2=0.02$.}
   \label{fig:_07} 
\end{subfigure}

\begin{subfigure}[b]{0.45\textwidth}
   \includegraphics[width=1\linewidth]{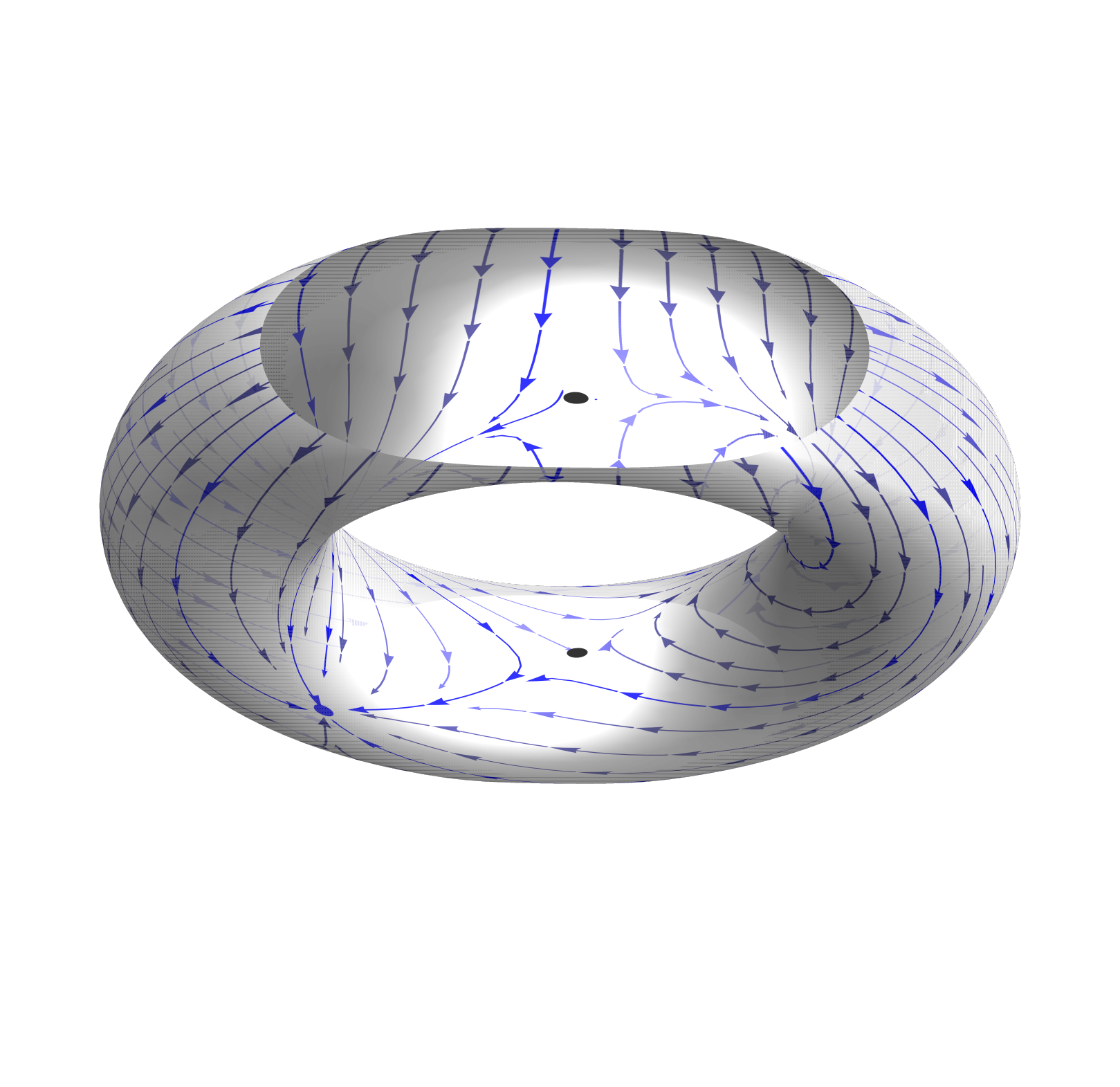}
   \caption{We can see a cut of the torus now for the perturbed dynamical system $\left( \mathbb{T}^2, \tilde{G} \right)$. Since the original dynamical system is structurally stable the phase portrait is only slightly distorted relative to the original}
   \label{fig:_08}
\end{subfigure}
\caption{The $\left( \mathbb{T}^2, \tilde{G} \right)$ flow in two perspectives. }
\end{figure}

\subsection{Equivariance}

Consider a linear bijection  $\Phi \in GL( \mathbb{T}^2)$ which commutes with $G$, that is 

\[
 F\circ\Phi \, (x,y)=\Phi\circ F \, (x,y) \ \ \ \forall (x,y) \in \mathbb{T}^2.
\]

Then, $G$ is  said to be a $\Phi$-equivariant map \cite{auslander,field1970,field1980,golubitsky2015}. The set of all such $\Phi$ is easily seen to form a group under composition. This group is a linear action of the symmetry group $\Gamma$ of the map $G$; with a slight abuse of language we identify this representation with $\Gamma$ itself, so that
\[ \Gamma = \{ \Phi \in GL(\mathbb{T}^2) : G\circ\Phi \, =\Phi\circ G  \}. \]

The following proposition summarizes some standard results in equivariant dynamics of which we shall make extensive use in the last section of this article; for completeness, we state and prove it in the present context. 
Recall that a set $S$ is strongly $G$-invariant if $G(S)=S$, while an orbit with initial condition $X_0$ is the set  $O_{X_{0}}= \{ \left(
X_{n}\right) \}_{n \in \mathbb{Z}}$ such that
\[
X_{n+1}=G\left(  X_{n}\right)  ,\quad n \in \mathbb{Z}.
\]

\begin{proposition}
Let $G: \mathbb{T}^2 \to \mathbb{T}^2$ be a $\Phi$-equivariant diffeomorphism. Then:
\begin{enumerate}
\item If the set $S$ is strongly $G$-invariant, then  the set $\Phi\left(  S\right)  $ is also strongly $G$-invariant.  
\item If $O_{X_{0}}$ is an orbit of $G$ with initial condition $X_{0}$, then
$
\Phi\left(  O_{X_{0}}\right)  = \{ (  \Phi (  X_{n}) \}_{n \in \mathbb{Z}}
$
is  an orbit of $G$ with initial condition $\Phi\left(  X_{0}\right)  $.
\end{enumerate}
\label{prop_invariance}
\end{proposition}

\begin{proof}
If $S$ is strongly invariant, that is $G(S) = S$, 
then $\Phi$-equivariance immediately implies 
\[
G\left(  \Phi\left(  S\right)  \right)  =\Phi\left(  G\left(  S\right)
\right)  =\Phi\left(  S\right) 
\]
showing invariance of $\Phi\left(  S\right)$ and proving the first statement. 

For the second statement, notice that $\Phi$-equivariance of $G$ implies 
\[
G^{n}\circ\Phi\left(  X\right)  =\Phi\circ G^{n}\left(  X\right)   \ \ \forall n \in \mathbb{Z}
\]
and therefore, if $Y_0 = \Phi(X_0)$, then 
\[
G^n(Y_0) = G^n(\Phi(X_0)) = \Phi(G^n(X_0)) \ \ \forall n \in \mathbb{Z}, \]
finishing the proof.
\end{proof}

Consider all orbits with initial conditions in an invariant set \( S \).  
Proposition \ref{prop_invariance} implies that any orbit in \( S \) has an equivalent orbit, in the sense of linear conjugacy, within the invariant set \( \Phi\left(S\right) \).  
More generally, the dynamics of each initial condition in \( S \) are linearly conjugate to the dynamics of the corresponding initial condition in \( \Phi\left(S\right) \).  
In other words, the flow of the dynamical system in \( S \) is linearly conjugate to the flow in \( \Phi\left(S\right) \).  

Naturally, the existence of a Lyapunov function in an open set \( S \) is equivalent to the existence of a Lyapunov function in the image of \( S \) under \( \Phi \), as shown in the next proposition. This result follows directly from Theorem \ref{thm_conjugacy}, since \( \Phi \) is also a topological conjugacy.

\begin{proposition}
\label{proposition:_eqVLyap} Consider a linear bijection $\Phi$ commuting with a diffeomorphism $G$, and a Lyapunov function $V$ for $G$, in some open set $S$ of $\mathbb{T}$. Then $V\circ \Phi^{-1} $ is a Lyapunov function in $R=\Phi S$.
\end{proposition}

The above proposition, despite its simplicity, is very useful in the study of dynamical systems arising from the iteration of diffeomorphisms with symmetries and will be used in the proof of Theorems~\ref{thm:_LyapunovG1} and \ref{thm:_LyapunovG2}. 

%\begin{figure}[ptb]
%\begin{center}
%\includegraphics[
%height=4.5in,
%width=4.5in
%]{zone}
%\end{center}
%\caption{We show here the set $D$ and its subdivisions in subsets $\closure{S}$ in light yellow dashed horizontally, $\overline{R}$ in light blue dashed vertically. The set $\Theta_1$ is shown in red.
%We can see as well the triangles $T_{1}$ and $T_{2}$ such that $\closure{S}=T_{1}\cup T_{2}$ and $T_{3}$ and $T_{4}$ such that $\overline{R}=T_{3}\cup T_{4}$. $S=\text{int}\closure{S}$ and $R=\text{int}\overline{R}$}%
%\label{fig:_03}%
%\end{figure}

\section{Technical details of the proof of Lyapunov Theorems}\label{sec_construction} 
In this section, we address the technical aspects of the proof of Theorem \ref{thm:_LyapunovG1} in the open set \( S \) mentioned in the theorem's statement.  
Our goal is to prove the negativeness of the orbital derivative of the Lyapunov function \( V \) within this open set. the proof of Theorem \ref{thm:_LyapunovG2} will then follow by symmetry considerations.  

We now return to the square \( D = \left[ 0, 2\pi \right] \times \left[ 0, 2\pi \right] \), as the analysis in this domain is equivalent but more convenient than on the torus \( \mathbb{T}^2 \). Additionally, this perspective simplifies the analysis along the edges of \( S \).

\subsection{Symmetry\label{Subsubsec:symmetry}}

The structure of the diffeomorphism \( G \) reveals several symmetries, which we will explore in the final part of the proof.  
We recall that the line \( y = x \) divides \( D \) into two invariant triangles, as all the edges of these triangles are heteroclinic connections or fixed points as noted in Remark~\ref{connections}. We denote these closed triangles by \( \overline{S} \) and \( \overline{R} \), defined as:
\[
\overline{S} = \left\{ (x, y) \in D : y \geq x \right\},
\]
\[
\overline{R} = \left\{ (x, y) \in D : y \leq x \right\}.
\]
These are closed triangles such that \( S = \operatorname{int}(\overline{S}) \) and \( R = \operatorname{int}(\overline{R}) \) where $S$ and $R$ are the open sets of the statements of Theorems~\ref{thm:_LyapunovG1} and \ref{thm:_LyapunovG2}.

The line $y=2\pi-x$ divides again the closed triangle
$\overline{S}$ in two new closed invariant  (again by Remark~\ref{connections}) triangles $T_{1}$ and $T_{2}$ defined by 
\[
T_{1}=\left\{  \left(  x,y\right)  \in \overline{S}:y\geq2\pi-x\right\},
\]%
\[
T_{2}=\left\{  \left(  x,y\right)  \in \overline{S}:y\leq2\pi-x\right\},
\]
as well as dividing the closed triangle $\overline{R}$ into two closed invariant (also by Remark~\ref{connections}) triangles 
$T_{3}$ and $T_{4}$ defined by 
\[
T_{3}=\left\{  \left(  x,y\right)  \in \overline{R}:y\leq2\pi-x\right\},
\]
\[
T_{4}=\left\{  \left(  x,y\right)  \in \overline{R}:y\geq2\pi-x\right\}  .
\]
This decomposition is shown in Fig.~\ref{fig:_04}, where we can also see the heteroclinics 
that separate  the various invariant sets. 

We have $D=\overline{S}\cup \overline{R}=\cup_{j=1}^{4}T_{j}$.

We next
 construct explicitly the elements of the linear symmetry group $\Gamma$ of the diffeomorphism $G$.

\begin{proposition} The following four maps commute with the diffeomorphism $G$:
\begin{enumerate}
\item  The identity map, denoted by $\Phi_1$:
\[%
\begin{array}
[c]{cccc}%
\Phi_1: & \mathbb{T}^{2} & \longrightarrow & \mathbb{T}^{2}\\
& \left[
\begin{array}
[c]{c}%
x\\
y
\end{array}
\right]  & \longmapsto & \left[
\begin{array}
[c]{c}%
x\\
y
\end{array}
\right]  \text{,}%
\end{array}
\]
\item The reflection  along the line $y=2\pi-x$, denoted by $\Phi_{2}$:
\[%
\begin{array}
[c]{cccc}%
\Phi_{2}: & \mathbb{T}^{2} & \longrightarrow & \mathbb{T}^{2}\\
& \left[
\begin{array}
[c]{c}%
x\\
y
\end{array}
\right]  & \longmapsto & \left[
\begin{array}
[c]{c}%
2\pi-y\\
2\pi-x
\end{array}
\right]  \text{,}%
\end{array}
\]
\item The rotation $\Phi_{3}$ by $\pi$ around $\left(  \pi,\pi\right)  $, denoted by $\Phi_{3}$:
\[%
\begin{array}
[c]{cccc}%
\Phi_{3}: & \mathbb{T}^{2} & \longrightarrow & \mathbb{T}^{2}\\
& \left[
\begin{array}
[c]{c}%
x\\
y
\end{array}
\right]  & \longmapsto & \left[
\begin{array}
[c]{c}%
2\pi-x\\
2\pi-y
\end{array}
\right]  \text{.}%
\end{array}
\]
\item The reflection along the line $y=x$, denoted by  $\Phi_{4}$:
\[%
\begin{array}
[c]{cccc}%
\Phi_{4}: & \mathbb{T}^{2} & \longrightarrow & \mathbb{T}^{2}\\
& \left[
\begin{array}
[c]{c}%
x\\
y
\end{array}
\right]  & \longmapsto & \left[
\begin{array}
[c]{c}%
y\\
x
\end{array}
\right]  \text{.}%
\end{array}
\]
\end{enumerate}

\label{prop_translations}
\end{proposition}

\begin{proof}The proof is, in each case, a simple computation.
\begin{enumerate}
\item The identity case is trivial.
\item For the reflection $\Phi_2$, we have
\begin{align*} 
G\left(  \Phi_2\left(  \left[
\begin{array}
[c]{c}%
x\\
y
\end{array}
\right]  \right)  \right)   &  =G\left(  \left[
\begin{array}
[c]{c}%
2\pi-y\\
2\pi-x
\end{array}
\right]  \right) \\
&  =\left[
\begin{array}
[c]{c}%
2\pi-y-2a\sin y-a\sin x - a\sin (x-y)\\
2\pi-x-a\sin y - 2a\sin x -a\sin(y-x)
\end{array}
\right] \\
&  =\left[
\begin{array}
[c]{c}%
2\pi\\
2\pi
\end{array}
\right]  -\left[
\begin{array}
[c]{c}%
g_2\left(  x,y\right) \\
g_1\left(  x,y\right)
\end{array}
\right] \\
&  =\Phi_{2}\left(  G\left(  \left[
\begin{array}
[c]{c}%
x\\
y
\end{array}
\right]  \right)  \right)  \text{.}%
\end{align*}
\item For the rotation $\Phi_3$, we have
\begin{align*}
G\left(  \Phi_3\left(  \left[
\begin{array}
[c]{c}%
x\\
y
\end{array}
\right]  \right)  \right)   &  =G\left(  \left[
\begin{array}
[c]{c}%
2\pi-x\\
2\pi-y
\end{array}
\right]  \right) \\
&  =\left[
\begin{array}
[c]{c}%
2\pi-x-2a\sin x-a\sin y - a\sin (x-y)\\
2\pi-y-a\sin x - 2a\sin y -a\sin(y-x)
\end{array}
\right] \\
&  =\left[
\begin{array}
[c]{c}%
2\pi\\
2\pi
\end{array}
\right]  -\left[
\begin{array}
[c]{c}%
g_1\left(  x,y\right) \\
g_2\left(  x,y\right)
\end{array}
\right] \\
&  =\Phi_{3}\left(  G\left(  \left[
\begin{array}
[c]{c}%
x\\
y
\end{array}
\right]  \right)  \right)  \text{.}%
\end{align*}
\item For the reflection $\Phi_4$, we have
\begin{align*}
G\left(  \Phi_4\left(  \left[
\begin{array}
[c]{c}%
x\\
y
\end{array}
\right]  \right)  \right)   &  =G\left(  \left[
\begin{array}
[c]{c}%
y\\
x
\end{array}
\right]  \right) \\
&  =\left[
\begin{array}
[c]{c}%
y+2a\sin y+a\sin x + a\sin (y-x)\\
x+a\sin y + 2a\sin x +a\sin(x-y)
\end{array}
\right] \\
&  =\left[
\begin{array}
[c]{c}%
g_2\left(  x,y\right) \\
g_1\left(  x,y\right)
\end{array}
\right] \\
&  =\Phi_{4}\left(  G\left(  \left[
\begin{array}
[c]{c}%
x\\
y
\end{array}
\right]  \right)  \right)  \text{.}%
\end{align*}
\end{enumerate}
\end{proof}

\begin{remark}
\label{rem_involutions}
Note that all the maps $\Phi_j$,  $j = 1, 2, 3, 4$, are involutions, that is self-inverses: $\Phi_j^{-1} = \Phi_j$.
\end{remark}

\begin{remark}
\label{rem_Gamma}
Incidentally, we remark that in the proofs below we will not need the full symmetry group $\Gamma$ but 
only the reflections.
\end{remark}

\subsection{Orbital derivative}

We
 now proceed to study negativeness of $\dot{V}$. The overarching strategy will be as follows.
We first partition \( \overline{S} \) into smaller, adequately chosen subsets. 
 For these subsets, we analyse the signs of the arguments  of the different terms 
in the orbital derivative to simplify expression \eqref{eq_V}, eliminating the absolute values.
Next, we estimate the actual value of the orbital derivative of \( V \) in each subset. Finally, we use the symmetries of \( G \) to extend the result to the entire open set \( S \).  

As we mentioned above, once Theorem~\ref{thm:_LyapunovG1} is proved, we obtain a very simple proof of Theorem \ref{thm:_LyapunovG2} by using symmetry and equivariance arguments.

We consider now the open set $S$, the compact set $H_1=\left\{ \left(  \frac{2\pi}{3},\frac{4\pi}{3}\right) \right\}$, and the Lyapunov function $V$ as defined in the statement of Theorem~\ref{thm:_LyapunovG1}. The
discrete orbital derivative $\dot{V}$ in the open set $S$ is  given by
\begin{equation}\label{eq:_ODerivative}
\begin{split}
\dot{V}(x,y)= &  -\left\vert 2\pi+x-2y\right\vert -\left\vert y-2x\right\vert
+\\
&  +\left\vert 2x-y+3a\sin x-3a\sin\left(  y-x\right)  \right\vert \\
&  +\left\vert 2\pi+x-2y -3a\sin y -3a\sin\left(  y-x\right)  \right\vert .
\end{split}
\end{equation}

\begin{claim}\label{claim:01}
The orbital derivative $\dot{V}$ is negative in $S\setminus H_1$, and therefore
the fixed point $ \left(  \frac{2\pi}{3},\frac{4\pi}{3}\right)$
is asymptotically stable.
\end{claim}

%From Remark~\ref{rem_param_region} we recall
%\begin{equation}\label{eq:_09}
%0<a<\frac{1}{3}
%\end{equation}
%throughout the
%rest of the analysis of the Lyapunov function.

\begin{definition}
\label{def_four functions}We define the four functions
\[%
\begin{split}
\psi_{1}\left(  x,y\right)  =  &  2\pi+x-2y\text{.}\\
\psi_{2}\left(  x,y\right)  =  &  y-2x\text{.}\\
\psi_{3}\left(  x,y\right)  =  &  2x-y+3a\sin x-3a\sin\left(  y-x\right)
\text{,}\\
\psi_{4}\left(  x,y\right)  =  &  2\pi+x-2y-3a\sin y-3a\sin\left(  y-x\right)
\text{,}%
\end{split}
\]

\end{definition}

With this definition, the orbital derivative \eqref{eq:_ODerivative} is written more compactly as 
\begin{equation}
\dot{V}\left(  x,y\right)  =
-\left\vert \psi_{1}\left(  x,y\right)  \right\vert 
-\left\vert \psi_{2}\left(  x,y\right)  \right\vert
+\left\vert \psi_{3}\left(  x,y\right)  \right\vert
+\left\vert \psi_{4}\left(x,y\right)  \right\vert 
.\label{eq_Vdot}%
\end{equation}

\begin{figure}[ptb]
\begin{center}
\includegraphics[height=4.1in,width=4.in]{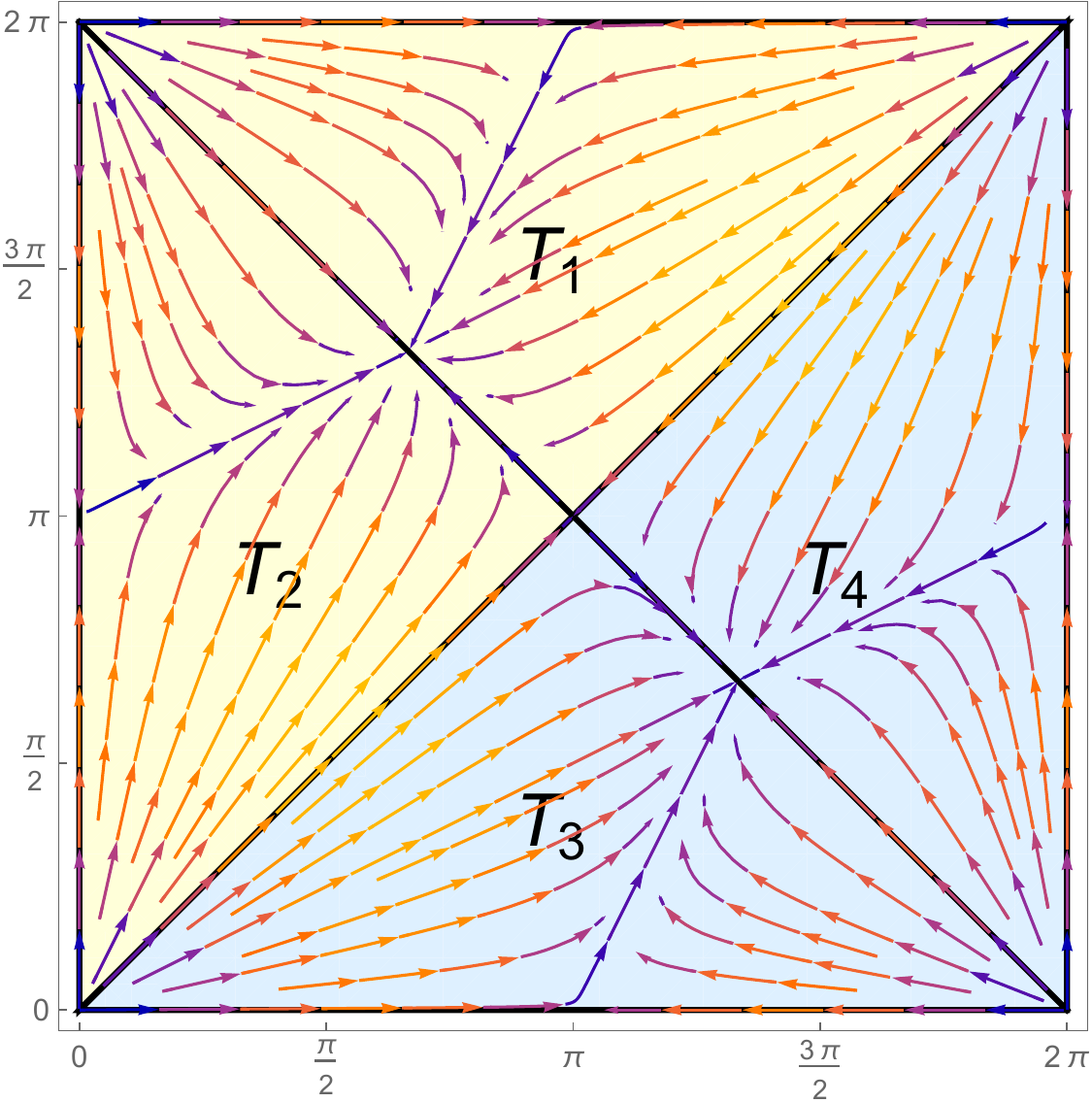}
\end{center}
\caption{We show here the set $D$, its subdivisions into subsets and the phase portrait showing heteroclinic connections. Our focus is on the light shaded region $\overline{S}=T_1 \cup T_2$ at the upper part of
the figure, as the darker region $\overline{R}=T_3 \cup T_4$ is related by symmetry to the former. The Lyapunov function $V$ is applied to the upper open triangle $S=\text{int} \overline{S}$ and its orbital derivative in that region is negative except at the fixed point $\left( \frac{2 \pi}{3}, \frac{4 \pi}{3} \right)$.}
\label{fig:_04}%
\end{figure}

We prove claim \ref{claim:01} in three steps:

\begin{enumerate}
\item We work in the triangle \( T_1 \), analyzing the signs of the individual terms \( \psi_i(x, t) \)  \( (i = 1, 2, 3, 4) \)  appearing in the expression for \( \dot{V} \) (equation \eqref{eq_Vdot}). This analysis allows us to drop the absolute value within each region where the signs remain constant in \( T_1 \).

\item We compute the sign of the orbital derivative $\dot{V}$ in
$T_{1}$.

\item Finally, we extend this analysis from $T_{1}$ to the triangle $\overline{S}$ by
symmetry arguments using $\Phi_{1}$, and draw conclusions about the dynamics in its interior, the
open set $S$.
\end{enumerate}

\subsection{Analysis of the signs of $\psi_{i}\left(  x,y\right)  $, $i=1,2,3,4$}

We now proceed with the first step of the strategy, analysing the signs of the
individual terms $\psi_{i}(x,y)$ in the triangle $T_{1}$. We separate the
analysis in a sequence of lemmas.

We split $T_{1}$ in three triangles seen in Fig.~\ref{fig:_09}.

\begin{definition}
For $(x,y)\in T_{1}$, we define
\[
T_{1}^{I}=\left\{  (x,y)\in T_{1}:y\geq 2 x\right\}  ,
\]%
\[
T_{1}^{II}=\left\{  (x,y)\in T_{1}:\pi+\frac{1}{2}x\leq y\leq2x\right\}  ,
\]%
\[
T_{1}^{III}=\left\{  (x,y)\in T_{1}:x\leq y\leq\pi+\frac{1}{2}x\right\}  .
\]

\end{definition}

\begin{remark}
As
noted in Remark~\ref{connections}, all the edges of the triangles  $\overline{S}$, $\overline{R}$, $T_{1}$,
$T_{2}$, $T_{3}$, $T_{4}$, $T_{1}^{I}$, $T_{1}^{II}$ and $T_{1}^{III}$ are heteroclinic connections and thus invariant segments for the dynamics of $G$.  All these triangles are compact sets, and therefore
continuous functions assume maxima and minima in those triangles and in their
 unions.
\end{remark}

We now proceed to study the variations of sign of $\psi_1$,  $\psi_2$, $\psi_3$ and
$\psi_4$ in the triangle $T_1$ through a sequence of lemmas.

\begin{lemma}[sign of $\psi_1$]
\label{lemma1} 
The regions within $T_1$ where $\psi_{1}$ has constant sign are:
\begin{enumerate}
\item $\psi_{1}\left(  x,y\right)  \leq 0$ for $\left(  x,y\right)  \in$ $T_{1}^{I}\cup T_{1}^{II}$;
\item 
$\psi_{1}\left(  x,y\right)  \leq0$ and for $\left(  x,y\right)  \in$
$T_{1}^{III}$.
\end{enumerate}

\end{lemma}

\begin{proof}
To prove (1) observe that, for 
 $\left(  x,y\right)  \in$ $T_{1}^{I}\cup T_{1}^{II}$,  we have
$\pi+\frac{1}{2}x\leq y\Leftrightarrow\pi+\frac{1}{2}x-y\leq0$ and therefore
$\psi_{1}\left(  x,y\right)  =2\pi+x-2y\leq0$ with equality holding exactly on the segment
$2\pi+x=2y$. To prove (2) observe that on the triangle $T_{1}^{III}$ 
we have $y\leq\pi+\frac{1}%
{2}x\Leftrightarrow\pi+\frac{1}{2}x-y\geq0$, and therefore $\psi_{1}\left(
x,y\right)  =2\pi+x-2y\geq 0.$
\end{proof}

\begin{lemma}[sign of $\psi_2$]
\label{lemma2}
The regions within $T_1$ where $\psi_{2}$ has constant sign are: 
\begin{enumerate}
\item $\psi_{2}\left( x,y\right)  \geq0$ for $\left(  x,y\right)  \in$ $T_{1}^{I}$;
\item 
$\psi_{2}\left(  x,y\right)  \leq0$ for $\left(  x,y\right)  \in$ $T_{1}^{II}\cup
T_{1}^{III}$.
\end{enumerate}
\end{lemma}

\begin{proof}
To prove (1) observe that, for $\left(  x,y\right)  \in$ $T_{1}^{I}$, we have $2x\leq y\Leftrightarrow y-2x \geq 0$, and therefore $\psi_{2}\left(  x,y\right)  =y-2x\geq0$ with equality holding only along the segment $2x=y$. 
To prove (2) note that on the triangle $T_{1}^{II}\cup T_{1}^{III}$ we have
$2x\geq y\Leftrightarrow  y-2x\leq 0$, and therefore $\psi_{2}\left(  x,y\right)
=y-2x\leq0.$
\end{proof}

\bigskip

\begin{figure}[ptb]
\includegraphics[
height=4in,
width=4in,center
]{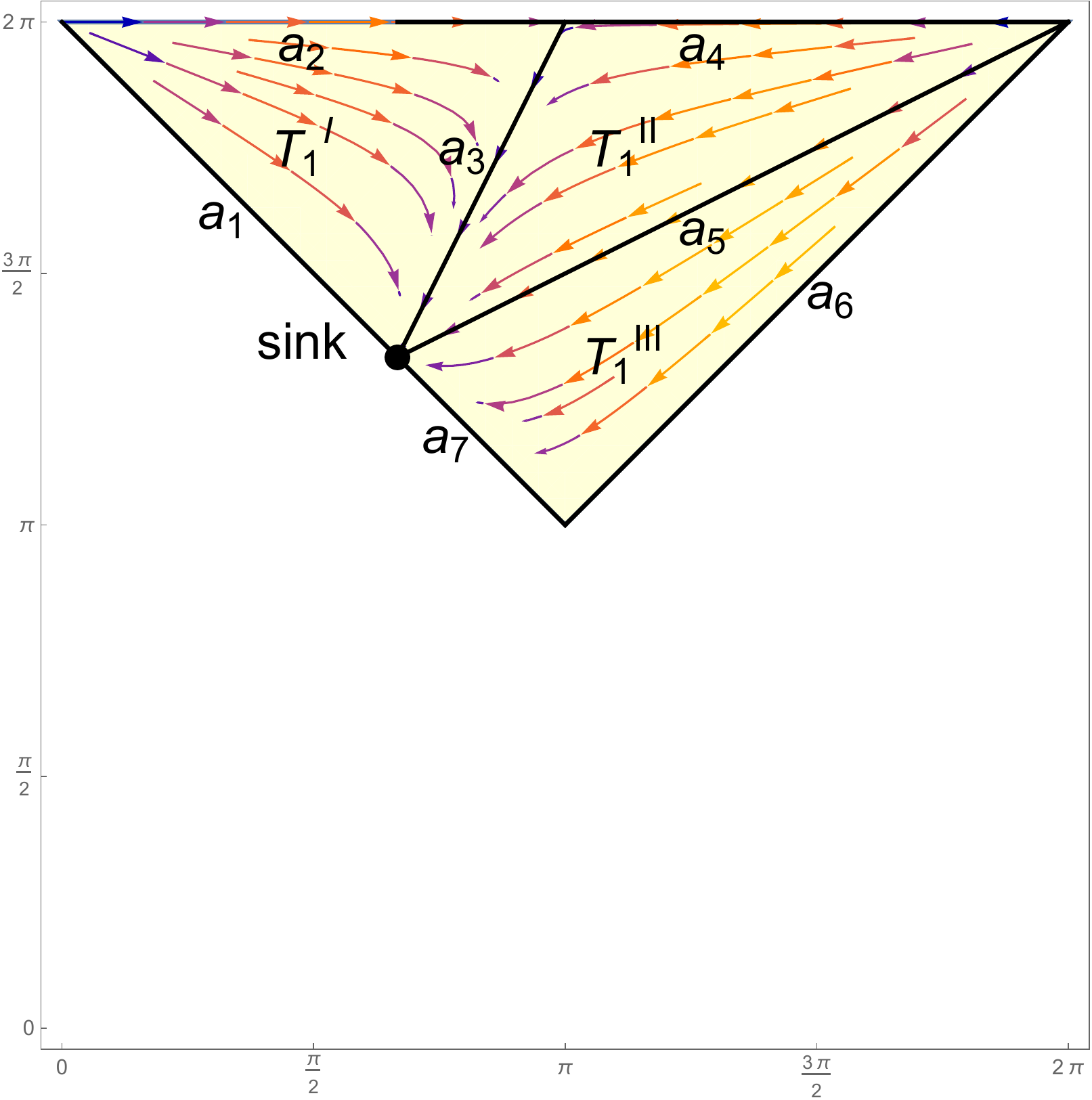}\caption{Detail of the triangle $T_1$ shadowed in light yellow subdivided in $T_1^I$, $T_1^{II}$ and $T_1^{III}$ with
their edges $a_j$, $j=1,\dots,7$}%
\label{fig:_09}%
\end{figure}

\begin{lemma}[sign of $\psi_3$]
\label{lemma3} The  
 regions within $T_1$ where $\psi_{3}$ has constant sign are:

\begin{enumerate}
\item $\psi_{3}\left(  x,y\right)  \leq0$ for $\left(  x,y\right)  \in$
$T_{1}^{I} $;

\item $\psi_{3}\left(  x,y\right)  \geq0$ for $\left(  x,y\right)  \in$
$T_{1}^{II}\cup T_{1}^{III}$.
\end{enumerate}
\end{lemma}

\begin{proof}
We have%
\[
\partial_{y}\psi_{3}\left(  x,y\right)  =-1-3a\cos\left(  y-x\right).
\]
Recalling from Remark~\ref{rem_param_region} that $0 < a < 1/3$, it follows that
 $\partial_{x}\psi_{3}\left(  x,y\right)  =0$ has no solutions and there are no local extrema.
 Therefore, maxima
and minima of $\psi_3$ must lie on the boundaries of $T_{1}^{I}$ and $T_{1}^{II}\cup
T_{1}^{III}$.

The 
common edge $a_3$  of the two polygons (see Fig.~\ref{fig:_09}) is described by the parametrization
\[a_3 = \{ (x,y) \in D: \frac{2\pi}{3} \leq x \leq \pi, \, y=2x \}, \]
on which $\psi_3$ is given by 
%over the line $y=2x$,
%$x\in\left[  \frac{2\pi}{3},\pi\right]  $ (see Fig.~\ref{fig:_09}), we have

\begin{equation*}\label{eq:_tau2}
\left. \psi_{3} \left( x, y \right)\right|_{y=2x} \equiv \xi_1\left( x \right)=2x-2x+3a\sin x-3a\sin\left(  2x-x\right)  \equiv 0.
\end{equation*}
This edge splits the triangle \( T_1 \) into two distinct sign regions for \( \psi_3 \), as we show below:
the triangle $T_{1}^{I}$ and the quadrilateral $T_{1}^{II}\cup T_{1}^{III}$.
\vspace{2mm}

 {\em The triangle $T_{1}^{I}$.}
%\begin{enumerate}
%\item 
We now analyse the two remaining edges of the triangle $T_{1}^{I}$.

\begin{enumerate}
\item \label{left edge}
The edge $a_1 $ is described by the parametrization
\[a_1 = \{ (x,y) \in D: 0 \leq x \leq \frac{2\pi}{3}, \, y=2\pi-x \}, \]
%over the line $y=2\pi-x$, for $x\in\left[
%0,\frac{2\pi}{3}\right]  $, where we have%
where the function $\psi_3$ is written 
 % \xi_1\left( x \right) 
\[
\left. \psi_{3} \left( x, y \right)\right|_{y=2\pi-x} \equiv \xi_2(x)  =3x-2\pi+3a\sin x+3a\sin2x \text{.}
\]
It is immediate to observe that $\xi_2(0) = -2\pi$ and $\xi_2(\frac{2\pi}{3}) = 0$, 
while its 
derivative is%
\[
\xi_2^{\prime}\left(  x \right)  =3+3a\cos x+3a\cos2x>0. 
%\text{ for }  0 \leq x \leq \frac{2\pi}{3}.
\]
This means that $\xi_2$ is strictly increasing with $x$, implying 
\begin{align*}
& \min_{(x,y) \in a_1} \psi_3(x,y) = \psi_3(0,2\pi) = -2\pi, \\ &\max_{(x,y) \in a_1} \psi_3(x,y) = 
\psi_3\left( \frac{2\pi}{3}, \frac{4\pi}{3}\right) = 0.
\end{align*}

%meaning that the minimum of $\psi_{3}$ is $-2\pi$ and the maximum is $0$ in
%this edge.

\item \label{top edge} The top edge $a_2$ is described by the parametrization
\[a_2 = \{ (x,y) \in D: 0 \leq x \leq \pi, \, y=2\pi\}, \]
where the function $\psi_3$ is written 
\[
\left. \psi_{3} \left( x, y \right)\right|_{y=2\pi} \equiv \xi_3(x)  =2x-2\pi+6a\sin x.
\]
It is immediate to observe that $\xi_3(0) = -2\pi$ and $\xi_3(\pi) = 0$, 
while its derivative is 
\[ \xi_3^\prime(x) = 2+6a\cos x> 0, \]
showing that $\xi_3$ is strictly increasing on this segment. This establishes, in 
a totally analogous fashion to the previous case, that 
\begin{align*}
&\min_{(x,y) \in a_2} \psi_3(x,y) = \psi_3(0,2\pi) = -2\pi, \\ 
&\max_{(x,y) \in a_2} \psi_3(x,y) = \psi_3(\pi, 2\pi) = 0.
\end{align*}

\end{enumerate}

From \eqref{left edge} and  \eqref{top edge} together with the fact that $\psi_3$ vanishes 
identically along the edge $a_3$ we conclude that 
\[ \max_{(x,y) \in T_{1}^{I}} \psi_3(x,y) = 0 \]
establishing that 
$\psi_{3}$ is non-positive on this closed triangle.

\vspace{2mm}

{\em The quadrilateral $T_{1}^{II}\cup T_{1}^{III}$.}

The analysis of the quadrilateral $T_{1}^{II}\cup T_{1}^{III}$ is 
similar, leading now to the conclusion that $\psi_{3}$ is non-negative. Since the edge $a_3$ 
is common to $T_{1}^{II}$ and has been analysed, we now concentrate on the 
remaining three edges $a_4, \, a_6$ and $a_7$.

\begin{enumerate}
\item The edge $a_4$ of $T_{1}^{II}\cup T_{1}^{III}$ is described by the parametrization
\[a_4 = \{ (x,y) \in D: \pi \leq x \leq 2\pi, \, y=2\pi\}, \]
and therefore on this edge
\[
\left. \psi_{3} \left( x, y \right)\right|_{y=2\pi} \equiv \xi_3(x)  =2x-2\pi+6a\sin x.
\]
\noindent The function $\psi_{3}$ is the same as in item \eqref{top edge} 
above, and similar calculations 
lead to the conclusion that  $\psi_{3}(\pi) =0$ and $\psi_{3}\left(  2\pi \right)  =2\pi$, with $\psi_{3}$ 
strictly increasing. This establishes that 
\begin{align*}
&\min_{(x,y) \in a_4} \psi_3(x,y) = \psi_3(\pi,2\pi) = 0, \\ &\max_{(x,y) \in a_4} \psi_3(x,y) = \psi_3(2\pi, 2\pi) = 2\pi. 
\end{align*}

%%%%%%%%%%%%%%%% AQUI A7

\item The edge $a_7$ of $T_{1}^{II}\cup T_{1}^{III}$ is described by the parametrization
\[a_7 = \{ (x,y) \in D: \frac{2\pi}{3} \leq x \leq \pi, \, y=2\pi -x\}, \]
which is the same line as the edge $a_1$ analysed above. Therefore, we have again 
\[
\left. \psi_{3} \left( x, y \right)\right|_{y=2\pi-x} \equiv \xi_2(x)  =3x-2\pi+3a\sin x+3a\sin2x \text{,}
\]
from which we conclude that $\xi_2 \left(\frac{2\pi}{3}\right) =0 $, $\xi_2(\pi) = \pi$ and 
$\xi_2$ is strictly increasing. Thus, we have
\begin{align*}
&\min_{(x,y) \in a_7} \psi_3(x,y) = \psi_3\left(\frac{2\pi}{3}, \frac{4\pi}{3}\right) = 0, \\& \max_{(x,y) \in a_7} \psi_3(x,y)  = \psi_3(\pi,\pi) = \pi.
\end{align*}

%%%%%%%%%%%%%%%%%%%%%%%% EDGE A6
\item Finally, the edge $a_6$  is described by the parametrization
\[a_6 = \{ (x,y) \in D: \pi \leq x \leq 2\pi, \, y=x\}, \]
where the function $\psi_3$ is written 
\[
\left. \psi_{3} \left( x, y \right)\right|_{y=x} \equiv \xi_4(x)=x+3a\sin x,
\]
whose derivative is%
\[
\xi_4^\prime(x)  =1+3a\cos x>0.
\]
Thus $\xi_4$ is strictly increasing in $a_6$ with minimum $\xi_4(\pi) = \pi$ 
and maximum $\xi_4(2\pi) = 2\pi$, from which we conclude
\begin{align*}
&\min_{(x,y) \in a_6} \psi_3(x,y) = \pi, \\ &\max_{(x,y) \in a_6} \psi_3(x,y)  = 2\pi.
\end{align*}
 
\end{enumerate}

We can conclude from the preceding discussion that, on the compact quadrilateral $T_{1}^{II}\cup T_{1}^{III}$,
the minimum value $0$ of $\psi_{3}$ is attained exactly in the entire edge $a_3$ with endpoints $(\frac{2\pi}{3}, \frac{4\pi}{3})$ and 
$(\pi, 2\pi)$.  Consequently, $\psi_{3}$ is
non-negative on the compact quadrilateral $T_{1}^{II}\cup T_{1}^{III}$.
%\end{enumerate}
\end{proof}

%To complete
%the analysis of the signs of the terms inside the absolute values
%in the orbital derivative, 
To complete 
the analysis, we now tackle the final term \[\psi_{4}\left(
x,y\right)  =2\pi+x-2y-3a\sin y-3a\sin\left(  y-x\right).  \]

\begin{lemma}[sign of $\psi_4$]
\label{lemma5} 
The regions within $T_1$ where $\psi_{2}$ has constant sign are:

\begin{enumerate}
\item $\psi_{4}\left(  x,y\right)  \leq0$ for $\left(  x,y\right)  \in$
$T_{1}^{I}\cup T_{1}^{II}$;

\item $\psi_{4}\left(  x,y\right)  \geq0$ for $\left(  x,y\right)  \in$
$T_{1}^{III}$.
\end{enumerate}
\end{lemma}

\begin{proof}
We have%
\[
\partial_{x}\psi_{4}\left(  x,y\right)  =1+3a\cos\left(  y-x\right).
\]
Recalling from Remark~\ref{rem_param_region} that $0 < a < 1/3$, it follows that
 $\partial_{x}\psi_{4}\left(  x,y\right)  =0$ has no solutions and there are no local extrema 
for $\psi_{4}$ in $T_1$. Therefore, maxima
and minima must lie on the boundaries of $T_{1}^{I}\cup T_{1}^{II}$ and
$T_{1}^{III}$. 

The edge $a_5$, common to the triangles $T_{1}^{II}$ and $ T_{1}^{III}$, is described by the parametrization
\[a_5 = \{ (x,y) \in D: \frac{2\pi}{3} \leq x \leq 2\pi, \, y=\pi + x/2\}. \]
On this line, we have
\begin{equation}\label{eq:_xi3}
\left. \psi_{4} \left( x, y \right)\right|_{y=\pi + x/2}   \equiv \xi_5\left( x \right)= -3a\sin\left(  \pi+\frac{1}%
{2}x\right)  -3a\sin\left(  \pi-\frac{1}{2}x\right)   \equiv 0.
\end{equation}
% The edge $a_7$ of $T_{1}^{II}\cup T_{1}^{III}$ is described by the parametrization
%\[a_7 = \{ (x,y) \in D: \frac{2\pi}{3} \leq x \leq \pi, \, y=2\pi -x\}, \]
%which is the same line as the edge $a_1$ analysed above. Therefore, we have again 
%\[
%\left. \psi_{3} \left( x, y \right)\right|_{y=2\pi-x} \equiv \xi_2(x)  =3x-2\pi+3a\sin x+3a\sin2x \text{,}
%\]
%from which we conclude that $\xi_2 \left(\frac{2\pi}{3}\right) =0 $, $\xi_2(\pi) = \pi$ and 
%$\xi_2$ is strictly increasing. Thus, we have
%\begin{align*}
%&\min_{(x,y) \in a_7} \psi_3(x,y) = \psi_3\left(\frac{2\pi}{3}, \frac{4\pi}{3}\right) = 0, \\& \max_{(x,y) \in a_7} \psi_3(x,y)  = \psi_3(\pi,\pi) = \pi.
%\end{align*}

%In the common edge $a_5$ of the two triangles over the line
%$y=\pi+\frac{1}{2}x$, $x\in\left[  \frac{2\pi}{3},2\pi\right]  $, we have 
%
%\begin{equation}\label{eq:_xi3}
%\psi_{4}\left(  x,\pi+\frac{1}{2}x\right)  =-3a\sin\left(  \pi+\frac{1}%
%{2}x\right)  -3a\sin\left(  \pi-\frac{1}{2}x\right)  =0.
%\end{equation}
\noindent Thus, this edge splits the triangle $T_{1}$ into two definite sign
regions for
 $\psi_{4}$: the triangle $T_{1}^{I}\cup T_{1}^{II}$ and the triangle $T_{1}^{III}$,
which we now proceed to analyse separately.

\vspace{2mm}
{\em The triangle $T_{1}^{I}\cup T_{1}^{II}$.} 

%\begin{enumerate}
\item We  now analyse the behaviour of $\psi_4$ on two remaining edges of the triangle $T_{1}^{I}\cup T_{1}^{II}$, referring to Fig.~\ref{fig:_09}.

\begin{enumerate}
\item \label{left edge1} 
The left edge $a_1$ is described by the parametrization
\[a_1 = \{ (x,y) \in D: 0 \leq x \leq \frac{2\pi}{3}, \, y=2\pi -x\}.  \]
On this line,  we have%
\[
\left. \psi_{4} \left( x, y \right)\right|_{y=2\pi -x}   \equiv \xi_2\left(x \right)=3x-2\pi+3a\sin x+3a\sin2x.
\]
Note that this is exactly the same function, over the same edge, studied in 
item \ref{left edge} in the proof of Lemma~\ref{lemma3}, and 
therefore we immediately conclude that
\begin{align*}
& \min_{(x,y) \in a_1} \psi_4(x,y) = \psi_4(0,2\pi) = -2\pi, \\ &\max_{(x,y) \in a_1} \psi_4(x,y) = 
\psi_4\left( \frac{2\pi}{3}, \frac{4\pi}{3}\right) = 0.
\end{align*}

%, since the function over the edge is the same, meaning that the
%minimum of $\psi_{4}$ is $-2\pi$ at $x=0$ and the maximum is $0$ at
%$x=\frac{2\pi}{3}$.

\item \label{top edge1}
The top edge  of $T_{1}^{I}\cup T_{1}^{II}$,  corresponding to $a_2 \cup a_4$ 
in fig. \ref{fig:_09}, is parametrized by
\[a_2 \cup a_4 = \{ (x,y) \in D: 0 \leq x \leq 2\pi, \, y=2\pi\}.  \]
On this edge we have
\[
\left. \psi_{4} \left( x, y \right)\right|_{y=2\pi}  
\equiv \xi_6(x) = x-2\pi+3a\sin x,
\]

with $\xi_6(0)= -2\pi$ and $ \xi_6(2\pi)=0$. We also have
\[
\xi_6^{\prime}(x)  =1+3a\cos x>0 \text{ for all }x\text{,}\]
so $\psi_{4}$ is strictly increasing with $x$ along this edge. Therefore
\begin{align*}
& \min_{(x,y) \in a_2 \cup a_4} \psi_4(x,y) = -2\pi, \\ &\max_{(x,y) \in a_2 \cup a_4} \psi_4(x,y)  = 0.
\end{align*}

%
%
%The situation is the same as in
%the previous edge, with a minimum $-2\pi$ at $x=0$ and a maximum $0$ at
%$x=2\pi$.
%

%we have
%$y=2\pi$, $x\in\left[  0,2\pi\right]  $, therefore%
%\[
%\psi_{4}\left(  x,2\pi\right)  =x-2\pi+3a\sin x,
%\]
%with $\psi_{4}\left(  0,2\pi\right)  =-2\pi$ and $\psi_{4}\left(  2\pi
%,2\pi\right)  =0$, its derivative is%
%\[
%\psi_{4}^{\prime}\left(  x,2\pi\right)  =1+3a\cos x>0\text{ for }x\in\left[
%0,\pi\right]\text{,}
%\]
%so $\psi_{3}$ is increasing in this segment. The situation is the same as in
%the previous edge, with a minimum $-2\pi$ at $x=0$ and a maximum $0$ at
%$x=2\pi$.
\end{enumerate}

From items \ref{left edge1} and \ref{top edge1}  above we conclude that 
\[ \max_{(x,y) \in T_{1}^{I}\cup T_{1}^{II}} \psi_4(x,y) =0, \]
implying that 
 $\psi_{4}$ is non-positive in this closed triangle and proving statement (1) in the lemma.

\vspace{2mm}
{\em The triangle $T_{1}^{III}$}. 

The analysis of triangle $T_{1}^{III}$ is performed in a similar fashion, leading now to the 
conclusion that $\psi_{4}$ is non-negative on this triangle. The edge $a_5$, common to 
$T_{1}^{I}\cup T_{1}^{II}$,  has been analised above with the conclusion that $\psi_4$
is identically zero along it, so it remains to consider the two edges $a_6$ and $a_7$, refer
to Fig.~\ref{fig:_09}.

\begin{enumerate}
\item The edge  $a_7$ is on the same line as edge $a_1$ studied in item \ref{left edge1} of the proof of this lemma
and is parametrized by 
\[a_7 = \{ (x,y) \in D: \frac{2\pi}{3} \leq x \leq \pi, \, y=2\pi -x \}. \]
%over the line $y=2\pi-x$, for $x\in\left[  \frac{2\pi}{3}%
%,\pi\right]  $, the same line as in the item \ref{left edge1} where we have
We have again%
\[
\left. \psi_{4} \left( x, y \right)\right|_{y=2\pi-x}  
\equiv \xi_2(x) =3x-2\pi+3a\sin x+3a\sin2x.
\]
We have $\psi_4(\frac{2\pi}{3}, \frac{4\pi}{3}) = 0$ and $\psi_4(\pi,\pi)=\pi$.
Since we have already shown in item \ref{left edge} of Lemma~\ref{lemma3} that this function 
is increasing as a function of $x$, we conclude that 
\begin{align*}
& \min_{(x,y) \in a_7 } \psi_4(x,y) = 0, \\ &\max_{(x,y) \in a_7 } \psi_4(x,y)  = \pi.
\end{align*}

% Since this
%function is increasing, as we have seen in item \ref{left edge} of lemma
%\ref{lemma3}, the minimum of $\psi_{4}$ is $0$ and the maximum is $\pi$ in this edge.

\item Finally, the edge $a_6$ of $T_{1}^{III}$ corresponds to the parametrization
\[a_6 = \{ (x,y) \in D: \pi \leq x \leq 2\pi, \, y=x \}. \]
Along this edge we have 
\[
\left. \psi_{4} \left( x, y \right)\right|_{y=x}  
\equiv \xi_7(x) =2\pi-x-3a\sin x,
\]
whose derivative is%
\[
\xi_7^{\prime}\left(  x\right)  =-1-3a\cos x<0\text{ for all }x \text{,}
\]
and thus $\psi_{4}$ is decreasing with $x$ along this edge, implying that
\begin{align*}
& \min_{(x,y) \in a_6 } \psi_4(x,y) = 0, \\ &\max_{(x,y) \in a_6 } \psi_4(x,y)  = \pi.
\end{align*}
%
%
%
%
% in this segment. There is a maximum $\pi$ at
%$x=\pi$ and a minimum $0$ at $x=2\pi$.
\end{enumerate}

Summing up the preceding analysis, we conclude that 
 the minimum $0$ of $\psi_{4}$ on the triangle $T_{1}^{III}$ is attained 
along the edge $a_5$, where $\psi_{4}$ vanishes identically, and that
$\psi_{4}$ is non-negative on the set $T_{1}^{III}$. This proves statement (2), 
concluding the proof of the lemma.
\end{proof}

\subsection{Sign of the orbital derivative}

We can now produce a table with the signs of the functions $\psi_{j}$,
$j=1,2,3,4$, in the three triangles $T_{1}^{I}$, $T_{1}^{II}$ and $T_{1}^{III}$.

\begin{table}[H]
\centering
\begin{tabular}
[c]{|c|c|c|c|c|c|}\hline
& $\psi_{1}\left(  x,y\right)  $ & $\psi_{2}\left(  x,y\right)  $ & $\psi
_{3}\left(  x,y\right)  $ & $\psi_{4}\left(  x,y\right)  $ & $\dot{V}\left(
x,y\right)  $\\\hline
$T_{1}^{I}$ & $\leq0$ & $\geq0$ & $\leq0$ & $\leq0$ & $\psi_{1}-\psi_{2}-\psi
_{3}-\psi_{4}$\\\hline
$T_{1}^{II}$ & $\leq0$ & $\leq0$ & $\geq0$ & $\leq0$ & $\psi_{1}+\psi_{2}%
+\psi_{3}-\psi_{4}$\\\hline
$T_{1}^{III}$ & $\geq0$ & $\leq0$ & $\geq0$ & $\geq0$ & $-\psi_{1}+\psi_{2}%
+\psi_{3}+\psi_{4}$\\\hline
\end{tabular}
\caption{Table of  signs of the $\psi_j$ ($j=1,2,3,4$) and corresponding expression for $\dot{V}$.}
\label{tabela}
\end{table}

\begin{theorem}
\label{Thm:_MainLyap}
The orbital derivative $\dot{V}$ in $T_1$ satisfies:
\begin{enumerate}
\item $\dot{V}(x,y) \leq 0$ for all $(x,y) \in T_1$, and
\item $\dot{V}(x,y) = 0$ only at the fixed points $\left(  \frac{2\pi}{3},\frac{4\pi}%
{3}\right)  $, $\left(  0,2\pi\right)  $,  $\left(  \pi,2\pi\right)$ and identically along the edge $a_6$.
\end{enumerate}

%The 
%orbital derivative is non-positive in $T_{1}$.
%Specifically, it is zero exactly at the edge\footnote{Which includes the fixed points $\left(  \pi,\pi\right)  $ and $\left(  2\pi,2\pi\right)  $} $a_6$ where $y=x$, for $x\in\left[  \pi
%,2\pi\right]  $ and at the fixed points $\left(  \frac{2\pi}{3},\frac{4\pi}%
%{3}\right)  $, $\left(  0,2\pi\right)  $ and $\left(  \pi,2\pi\right)  $ being negative elsewere in $T_1$.
\end{theorem}

\begin{proof}
In 
order to compute the orbital derivative $\dot{V}$ we again 
split $T^{1}$ into  the three triangles $T_{1}^{I},T_{1}^{II}$ and $T_{1}^{III}$ (refer once more to fig. \ref{fig:_09}). The results of  Lemmas \ref{lemma1} to \ref{lemma5}, summarized in  Table \ref{tabela}, allow us to eliminate the absolute values in expression \eqref{eq_Vdot}, giving rise to the following three cases. %, which we examine separately.

\vspace{2mm}
{\em The triangle $T_{1}^{I}$.}

%\begin{enumerate}
%\item {} 
In triangle $T_{1}^{I}$ the expression for the orbital derivative
is
\[
\begin{aligned}
\dot{V}(x,y) &  =\psi_{1}-\psi_{2}-\psi_{3}-\psi_{4}\\
&  =-3a\sin x+3a\sin y+6a\sin\left(  y-x\right)  .
\end{aligned}
\]
There are no local extrema for $\dot{V}(x,y)$ in the interior of $T_{1}^{I}
$ since the stationarity equations
\begin{align*}
\partial_{x}\dot{V}(x,y)  &  =-3a\cos x-6a\cos\left(  y-x\right)  =0,\\
\partial_{y}\dot{V}(x,y)  &  =3a\cos y+6a\cos\left(  y-x\right)  =0,
\end{align*}
only have solutions for $y=2\pi-x$, i.e., on  edge $a_1$ of this
triangle. 
Therefore the maximum and minimum values of $\dot{V}(x,y)$ on  the compact set $T_{1}^{I}$ occur at its edges.
%can occur only at the edges of the compact set $T_{1}^{I}$. 

We next examine each of the three edges of $T_{1}^{I}$.

\begin{enumerate}
\item The edge $a_1$ corresponds to the parametrization
\[
a_1 = \{ (x,y) \in D: 0 \leq x \leq \frac{2\pi}{3}, \, y=2\pi-x \}.
\]
The orbital derivative on $a_1$ is
\begin{align*}
\left. \dot{V}\left( x, y \right)\right|_{y=2\pi-x}  
 & =-6a\sin x-6a\sin\left(  2x\right) \\
&  =-6a\sin x\left(  1+2\cos x\right).
\end{align*}
Thus along this edge $\dot{V}$ vanishes exactly at $x=0$ and $x=\frac{2\pi}{3}$, corresponding to the fixed points of the map $(0,2\pi)$ and 
$(\frac{2\pi}{3}, \frac{4\pi}{3}) $, and is strictly negative for $0 < x < \frac{2\pi}{3}$. 
%with roots at $x=0$ and $x=\frac{2\pi}{3}$, and trivially negative between those roots
%for $x\in\left]  0,\frac{2\pi}{3} \right[  $. 
%Therefore we have
%\[ \min_{(x,y) \in a_1} \dot{V}(x,y)  \equiv m_1 < 0. \]
%there exists a negative minimum
%$\min_{1}$ in the edge under consideration, the actual value of this negative minimum is immaterial in our reasoning.

\item The edge $a_2$ corresponds to the parametrization
\[
a_2 = \{ (x,y) \in D: 0 \leq x \leq \pi, \, y=2\pi \}.
\]
The orbital derivative on $a_2$ is
\begin{align*}
\left. \dot{V}\left( x, y \right)\right|_{y=2\pi}=-9a\sin x,
\end{align*}
which vanishes only at $x=0$ and $x=\pi$ (corresponding to the fixed points $(0, 2\pi)$ and $(\pi, 2\pi)$)
and is strictly negative for $0 < x < \pi$. 
%
%negative between
%those roots for $x\in\left]  0,\pi\right[  $. Therefore, there
%exists a negative minimum $\min_{2}$ in the edge under consideration.

\item The edge $a_3$ corresponds to the parametrization
\[
a_3 = \{ (x,y) \in D: \frac{2\pi}{3} \leq x \leq \pi, \, y=2 x \}.
\]
The orbital derivative on $a_3$ is given by
\begin{align*}
\left. \dot{V}\left( x, y \right)\right|_{y=2 x}  
&  =3a\sin x+3a\sin2x\\
&  =3a\sin x\left(  1+2\cos x\right),
\end{align*}
which vanishes only at $x=\frac{2\pi}{3}$ and $x=\pi$ (corresponding to the fixed points $(\frac{2\pi}{3}, \frac{4\pi}{3}) $ and $(\pi, 2\pi)$) and is strictly negative 
for $ \frac{2\pi}{3} < x < \pi$. 
%
%with roots at $x=\frac{2\pi}{3}$ and $x=\pi$, and trivially negative between
%those roots for $x\in\left]  \frac{2\pi}{3},\pi\right[  $. Therefore, there
%exists a negative minimum $\min_{3}$ in the edge under consideration.
\end{enumerate}
From the above analysis we conclude that $\dot{V}((x,y)) \leq  0$ for all  
$(x,y) \in T_{1}^{I}$, with equality attained exactly at the vertices 
%
%We just concluded that in $T_{1}^{I}$ the orbital derivative $\dot{V}(x,y)$
%has a maximum, which is zero, attained in the vertexes 
of $T_{1}^{I}$, i.e. the fixed points $\left(  0,2\pi\right)  $, $\left(  \frac{2\pi}{3}%
,\frac{4\pi}{3}\right)  $ and $\left(  \pi,2\pi\right)$, and strict inequality holding everywhere else in $T_{1}^{I}$.

%
%, with equality being attained exactly at the ver
%\[ \max_{(x,y) \in T_{1}^{I}}  \dot{V}((x,y)) \leq  0 \]
%and this maximum value is attained exactly at the vertices 
%%
%%We just concluded that in $T_{1}^{I}$ the orbital derivative $\dot{V}(x,y)$
%%has a maximum, which is zero, attained in the vertexes 
%of $T_{1}^{I}$, i.e. the fixed points $\left(  0,2\pi\right)  $, $\left(  \frac{2\pi}{3}%
%,\frac{4\pi}{3}\right)  $ and $\left(  \pi,2\pi\right) $. Elsewhere in
%$T_{1}^{I}$ \ the orbital derivative is negative.

\vspace{2mm}
{\em The triangle $T_{1}^{II}$.}

It follows from Table \ref{tabela} that in triangle $T_{1}^{II}$ the expression of the orbital derivative is 
\begin{align*}
\dot{V}(x,y)  &  =\psi_{1}+\psi_{2}+\psi_{3}-\psi_{4}\\
&  =3a\sin x+3a\sin y.
\end{align*}
The corresponding stationarity equations
\begin{align*}
\partial_{x}\dot{V}(x,y)  &  =3a\cos x=0,\\
\partial_{y}\dot{V}(x,y)  &  =3a\cos y=0,
\end{align*}
have no solutions in $T_{1}^{II}$, and therefore there are no local extrema.  The extrema of $\dot{V}$ in $T_{1}^{II}$ occur on 
the edges of this triangle.

\begin{enumerate}
\item The edge $a_3$ is 
common to $T_{1}^{I}$ and was analysed in (3) above, 
where it was shown that $\dot{V} (\frac{2\pi}{3}, \frac{4\pi}{3}) = 
\dot{V} (\pi, 2\pi) = 0$ and $\dot{V}$  is strictly negative along the edge $a_3$ connecting these two vertices. 
%for $ \frac{2\pi}{3} < x < \pi$.
%
%
%
%
%with negative minimum
%$\min_{3}$ and maximum $0$ attained in both endpoints.

\item The top edge $a_4$ is described by the parametrization
\[
a_4 = \{ (x,y) \in D: \pi \leq x \leq 2\pi, \, y=2\pi \}.
\]
Along this edge we have%
\[
\left. \dot{V}\left( x, y \right)\right|_{y=2\pi}=3a\sin x,
\]
which is strictly negative along the edge except at the endpoints,
where $\dot{V}(\pi, 2\pi) = \dot{V}(2\pi, 2\pi) =0$.
% negative, except at the end points where attains its
%maximum, which is zero, the negative minimum is $\min_{4}$.

\item The edge $a_5$ is described by the parametrization
\[
a_5 = \left\{ (x,y) \in D: \frac{2\pi}{3} \leq x \leq 2\pi, \, y=\pi-\frac{1}{2}x \right\}.
\]
Along this edge we have%
\begin{align*}
\left. \dot{V}\left( x, y \right)\right|_{y=\pi-\frac{1}{2}x}& =3a\sin x-3a\sin\frac{x}{2}\\
&  =3a\sin\frac{x}{2}\left(  -1+2\cos\frac{x}{2}\right),
\end{align*}
which is strictly negative except at the endpoints of $a_5$, where 
$\dot{V}(\frac{2\pi}{3}, \frac{4\pi}{3}) = \dot{V}(2\pi, 2\pi) =0$.

% $x=\frac{2\pi}{3}$ and $x=2\pi$ where the orbital derivative attains its maximum, zero.
%% Somewere in
%this edge the orbital derivative attains its minimum which we call $\min
%_{5}$.
\end{enumerate}

From the above analysis we conclude that $\dot{V}((x,y)) \leq  0$ for all  $(x,y) \in T_{1}^{II}$, with equality attained exactly at the vertices 
of $T_{1}^{II}$, i.e. the fixed points 
 $\left(  \pi,2\pi\right)  $, $\left(  \frac{2\pi}{3},\frac{4\pi
}{3}\right)  $ and $\left(  2\pi,2\pi\right)  $, 
%$\left(  0,2\pi\right)  $, $\left(  \frac{2\pi}{3}%
%,\frac{4\pi}{3}\right)  $ and $\left(  \pi,2\pi\right)$, 
and strict inequality holding everywhere else in $T_{1}^{II}$.

%
%
%
%
%
%We concluded that in $T_{1}^{II}$ the orbital derivative $\dot{V}(x,y)$ has a
%maximum, which is zero, attained in the vertexes of $T_{1}^{II}$ which are the
%fixed points $\left(  \pi,2\pi\right)  $, $\left(  \frac{2\pi}{3},\frac{4\pi
%}{3}\right)  $ and $\left(  2\pi,2\pi\right)  $. Elsewere in $T_{1}^{II}$
%\ the orbital derivative is negative.

\vspace{2mm}
{\em The triangle $T_{1}^{III}$.}

It follows from Table \ref{tabela} that in triangle $T_{1}^{III}$ the expression of the orbital derivative is 
%In the triangle $T_{1}^{III}$, the orbital derivative is
\begin{align*}
\dot{V}(x,y)  &  =-\psi_{1}+\psi_{2}+\psi_{3}+\psi_{4}\\
&  =3a\sin x-3a\sin y-6a\sin\left(  y-x\right).
\end{align*}
There are no local extrema for $\dot{V}(x,y)$ in the interior of
$T_{1}^{III}$, since the stationarity system
\begin{align*}
\partial_{x}\dot{V}(x,y)  &  =3a\cos x+6a\cos\left(  y-x\right)  =0,\\
\partial_{y}\dot{V}(x,y)  &  =-3a\cos y-6a\cos\left(  y-x\right)  =0,
\end{align*}
can have solutions only in the bottom edge $a_6$, where $y=x$ and $\pi \leq x \leq 2\pi $. 

We now analyse the three edges of $T_{1}^{III}$.

\begin{enumerate}
\item The edge $a_5$ is shared with $T_{1}^{III}$ and was studied above, with the conclusion that $\dot{V}$ is strictly negative except at the endpoints of $a_5$, where $\dot{V}(\frac{2\pi}{3}, \frac{4\pi}{3}) = \dot{V}(2\pi, 2\pi) =0$.  

 %with the same conclusions and the negative minimum $\min_{5}$.
\item The edge $a_6$ is 
described by the parametrization
\[
a_6 = \{ (x,y) \in D: \pi\leq x \leq 2\pi, \, y=x \},
\]
where we have trivially%
\[
\left. \dot{V}\left( x, y \right)\right|_{y=x} \equiv 0\text{.}%
\]
\item The edge $a_7$ is described by the parametrization
\[
a_7 = \{ (x,y) \in D: \frac{2\pi}{3}\leq x \leq \pi, \, y=2\pi-x \},
\]
along which we have 
\begin{align*}
\left. \dot{V}\left( x, y \right)\right|_{y=2\pi-x} &=6a\sin x+6a\sin\left(  2x\right) \\
&  =6a\sin x\left(  1+2a\cos x\right)  ,
\end{align*}
which again is strictly negative for $ \frac{2\pi}{3} < x < \pi $
with maxima at the endpoints, 
$\dot{V}(\frac{2\pi}{3}, \frac{4\pi}{3}) = \dot{V}(\pi, \pi) =0$.  
% with minimum $\min_{6}$ and maximum zero at the endpoints.
\end{enumerate}

From the above analysis we conclude that $\dot{V}((x,y)) \leq  0$ for all  $(x,y) \in T_{1}^{III}$, with equality attained exactly at the fixed point $\left(  0,2\pi\right)  $, $\left(  \frac{2\pi}{3} ,\frac{4\pi}{3}\right)  $ and identically along the edge $a_6$
 and strict inequality holding everywhere else in $T_{1}^{III}$.

Collecting all the results above and recalling 
that $T_1 = T_{1}^{I} \cup T_{1}^{II} \cup T_{1}^{III} $, we 
conclude that 
the orbital derivative $\dot{V}$ is strictly negative in $T_1$ except 
at the fixed points $\left(  \frac{2\pi}{3},\frac{4\pi}
{3}\right)  $, $\left(  0,2\pi\right)  $,  $\left(  \pi,2\pi\right)$ and identically along the edge $a_6$, where it attains its maximum value 0. This concludes the proof of Theorem \ref{Thm:_MainLyap}.

%
%
%
%We concluded that in $T_{1}^{III}$ the orbital derivative $\dot{V}(x,y)$ has a
%maximum, which is zero, attained in the vertex $\left(  \frac{2\pi}{3}%
%,\frac{4\pi}{3}\right)  $ and in the edge $y=x$, for $x\in\left[  \pi
%,2\pi\right]  $. Elsewere in $T_{1}^{III}$ the orbital derivative is negative.

%\end{enumerate}
\end{proof}

\subsection{Full picture\label{par:full}}

We have proved that, in the compact triangle $T_{1}$, the orbital derivative 
$\dot{V}$
is non-positive, being strictly negative in ${\rm int}(T_{1})$.
We now extend the
previous results to the open set $S$ containing the fixed point $(\frac{2\pi
}{3},\frac{4\pi}{3})$, finalizing the proof of claim \ref{claim:01}. 

The triangle $T_{2}$ is the reflection of the triangle $T_{1}$ along the line
$y=2\pi-x$ (see fig. \ref{fig:_04}). From Proposition~\ref{prop_translations} we obtain 
$T_{2}=\Phi_{2}(T_{1})$. %Also note that $\overline{S}=T_{1}\cup T_{2}$.

%The union
%\[
%%
%\]
%is a compact set such that $S=$int$\overline{S}$.

\begin{lemma}
\label{lemma:_eqV} We have
\begin{equation}
V\left(  \Phi_{2}\left(  x,y\right)  \right)  =V\left(  x,y\right)  .
\end{equation}

\end{lemma}

\begin{proof}
Recalling the definition of $V$ from \eqref{eq_V}
and the definition of $\Phi_{2}$ from \ref{prop_translations}, this is a simple verification. We have
\begin{align*}
V\left(  \Phi_{2}\left(  x,y\right)  \right)  = &V\left(  2\pi-y,2\pi-x\right)\\
= &\left\vert -2\pi-x+2y\right\vert +\left\vert -y+2x\right\vert \\
= &V\left(
x,y\right)  .
\end{align*}

\end{proof}

\begin{lemma}
\label{lemma:_eqOD} We have
\begin{equation}
\dot{V}\left(  \Phi_{2}\left(  x,y\right)  \right)  =\dot{V}\left(
x,y\right).
\end{equation}

\end{lemma}

\begin{proof}
By 
definition
\[
\dot{V}\left(  \Phi_{2}\left(  x,y\right)  \right)  =V\left(  G\circ\Phi
_{2}\left(  x,y\right)  \right)  -V\left(  \Phi_{2}\left(  x,y\right)
\right).
\]
Recalling from Proposition~\ref{prop_translations} that $G$ and $\Phi_{2}$ commute, we obtain 
\[
\dot{V}\left(  \Phi_{2}\left(  x,y\right)  \right)  =V\left(  \Phi_{2}\circ
G\left(  x,y\right)  \right)  -V\left(  \Phi_{2}\left(  x,y\right)  \right),
\]
from which, applying now Lemma~\ref{lemma:_eqV}, we conclude that
\[
\dot{V}\left(  \Phi_2\left(  x,y\right)  \right)  =V\left(  G\left(  x,y\right)
\right)  -V\left(  \left(  x,y\right)  \right)  =\dot{V}\left(  x,y\right)  .
\]

\end{proof}

Lemma~\ref{lemma:_eqOD} implies that the orbital derivative is non-positive in
$T_{2}$ if and only if it is non-positive in $T_{1}$. Recalling that $S$ is an open set and that $\overline{S}=T_{1}\cup T_{2}$, we conclude:

\begin{proposition}
The orbital derivative $\dot{V}\left(  x,y\right)  $ is negative in $S$ 
%the
%interior of
%\[
%\overline{S}=T_{1}\cup T_{2},
%\]
except at $(\frac{2\pi}{3},\frac{4\pi}{3})$, where it vanishes.
\label{thm_negative}
\end{proposition}

\begin{proof}

By Theorem~\ref{Thm:_MainLyap}, on the boundary segment separating $T_{1}$
and $T_2$ along the line $y=2\pi-x$ 
%common with $T_{2\text{}}$  and belonging to the interior of $\overline{S}$ 
the orbital derivative is negative, except at its endpoints which coincide with the fixed points $\left(
0,2\pi\right)$,  $\left(  \pi,\pi\right)$ and $(\frac{2\pi}{3},\frac{4\pi}{3})$, of which only the last one is in the open set $S$. In $S \setminus \{ (\frac{2\pi}{3},\frac{4\pi}{3})\} $ the orbital derivative is negative as a consequence of Theorem~\ref{Thm:_MainLyap}  combined with Lemma~\ref{lemma:_eqOD}.
\end{proof}

Proposition
 \ref{thm_negative} summarizes all the results in this section so far, establishing negativeness of the Lyapunov function $V$ in $S$ except at the set $H=\left\{  (\frac{2\pi}{3} 
,\frac{4\pi}{3})\right\}  $. This concludes the proof of Theorem~\ref{thm:_LyapunovG1}.
 
The situation with respect to the bottom open triangle $R$ corresponding to Theorem~\ref{thm:_LyapunovG2} is very similar. In fact, Theorem~\ref{thm:_LyapunovG2} follows from Theorem~\ref{thm:_LyapunovG1} using the symmetry of the system.

\begin{proof}[Proof of Theorem~\ref{thm:_LyapunovG2}.]
Recall that $\Phi_4$ is an involution, that is, 
 $\Phi_{4}^{-1}=\Phi_4$. The Lyapunov function $U$ for $R$ satisfies
\[
U\left(  x,y\right)  =V\circ\Phi_{4}\left(  x,y\right),
\]
or
\[
U\left(  x,y\right)  =V\circ\Phi_{4}\left(  x,y\right)  =V\left(  y,x\right)  =\left\vert
x-2y\right\vert +\left\vert 2\pi+y-2x\right\vert \text{.}%
\]

\noindent Let $\left(  x,y\right)  \in R$ and $\left(  X,Y\right)  =\Phi_{4}\left(  x,y\right) \in S$. We have 
%We have $\left(  X,Y\right)  =\Phi_{4}\left(  x,y\right)  $ and $\left(
%X,Y\right)  \in S$ iff $\left(  x,y\right)  \in R$, moreover 
\[ \Phi_{4}%
\left(\frac{4\pi}{3},\frac{2\pi}{3}\right)=\left(\frac{2\pi}{3},\frac{4\pi}{3}\right).\] 
The orbital
derivative of $U$ on $R$ is therefore 
\begin{align*}
V\circ\Phi_{4}\circ G\left(  x,y\right)  -V\circ\Phi_{4}\left(  x,y\right)   &
=V\circ G\circ\Phi_{4}\left(  x,y\right)  -V\circ\Phi_{4}\left(  x,y\right) \\
& =V\circ G\circ\left(  X,Y\right)  -V\left(  X,Y\right)  \leq 0.
\end{align*}
It thus follows from Proposition~\ref{thm_negative} that the orbital derivative $\dot{U}$
 strictly negative in $R \setminus \{(\frac{4\pi}{3},\frac{2\pi}{3}) \} $ and zero
at $(\frac{4\pi}{3},\frac{2\pi}{3})$. This concludes the proof.
% and strictly negative in 
%$R \setminus \{(\frac{4\pi}{3},\frac{2\pi}{3}) \} $.
\end{proof}

\section{Conclusion}\label{sec:_conclusion}
\subsection{The synchronisation diffeomorphism for nearest-neighbour interaction.}

In a recent paper \cite{BAO2025}, 
the authors constructed a Lyapunov function $V$
for the diffeomorphism of the torus $\mathbb{T}^2$
\begin{align}\label{eq:_F1}
\begin{bmatrix}
x_{n+1}\\
y_{n+1}%
\end{bmatrix}
=F
\begin{bmatrix}
x_{n}\\
y_{n}%
\end{bmatrix}
=%
\begin{bmatrix}
x_{n}+2a\sin x_{n}+a\sin y_{n}\\
y_{n}+a\sin x_{n}+2a\sin y_{n}
\end{bmatrix}.
\end{align}
with components 
\begin{equation}
\label{eq_F}
F(x,y)=\left( f_1\left(x,y \right), f_2\left(x,y\right) \right)
\end{equation} 

\noindent modelling  
the nearest-neighbour Huygens interaction of three clocks on a line.
It was shown that the diffeomorphism $F$ is Morse-Smale and has a unique hyperbolic attractor \( \left\{ \left(\pi , \pi \right) \right\} \), whose basin of attraction is an open set $S$ such that $\bar{S}=\mathbb{T}^2$. 

%The Lyapunov function we consider is
%\begin{equation}
%\label{function_V}
%V\left( x,y\right) =\left\vert \pi-y\right\vert +\left\vert \pi-x\right\vert.
%\end{equation}

Consider 
now the perturbed diffeomorphism $\tilde{F}$

\begin{align}\label{eq:_F2}
\begin{bmatrix}
x_{n+1}\\
y_{n+1}%
\end{bmatrix}
=\tilde{F}
\begin{bmatrix}
x_{n}\\
y_{n}%
\end{bmatrix}
=%
\begin{bmatrix}
x_{n}+2a\sin x_{n}+a\sin y_{n} + \delta_{1} \zeta_1(x,y) \\
y_{n}+a\sin x_{n}+2a\sin y_{n} + \delta_{2} \zeta_2(x,y)
\end{bmatrix}.
\end{align}
with components 
\begin{equation}
\label{eq_F2}
\tilde{F}(x,y)=\left( \tilde{f}_1\left(x,y \right), \tilde{f}_2\left(x,y\right) \right).
\end{equation}

\noindent All the theory described in section \ref{sec_intro} applies in this context; namely, 
for sufficiently small $\epsilon = |\delta_1| + |\delta_2|$, the diffeomorphisms $F$ and $\tilde{F}$ are topologically conjugate. Denoting such a conjugacy by  $h$,  we have the following result.

\begin{theorem}
\label{thm_Lyap_perturbed2}
Let $\left( \mathbb{T}^2, F \right)$ and $\left( \mathbb{T}^2, \tilde{F} \right)$ be as above, and let $h: \mathbb{T}^2 \to \mathbb{T}^2$ be a topological conjugacy. Then
$h\left( \pi,\pi \right)$ is a sink for $\tilde{F}$ with a strict Lyapunov function $\tilde{V} = V \circ h^{-1}$ on the open set $h(S)$.
\end{theorem}

This result implies that, as happens in the case of oscillators arranged in a ring studied in the present paper, the synchronisation phenomenon in the nearest-neighbour  model is structurally stable.
For the unperturbed system synchronisation occurs in a single,  unique state 
corresponding to phase opposition. This implies that the perturbed systems, corresponding to non-identical clocks, will synchronise will probability $1$ near phase opposition between consecutive oscillators.

\subsection{General conclusion}

Lyapunov functions, introduced well over a century ago, remain an essential
tool for analyzing the stability of dynamical systems, in both theoretical and
practical contexts, across science and engineering. Constructing these
functions is an ongoing challenge that impacts various fields, from real-world
engineering applications to mathematics proper. The discovery of a dynamical
system admitting an explicit Lyapunov function may thus be considered a
striking situation.

The diffeomorphisms \eqref{eq:_4} and \eqref{eq_F2}, arising in the problem of synchronisation of
three limit cycle oscillators were studied in \cite{BAO2024b,BAO2025,EH2} and shown to have two sinks and one sink respectively. This was done by constructing a network of heteroclinic connections and
showing laboriously that each fixed point is asymptotically stable and that
their basin of attraction is the interior of the region bounded by heteroclinics as well constructing a Lyapuniov function for the diffeomorphism \eqref{eq_F2}.

In this paper we prove asymptotic stability of the fixed points of \eqref{eq:_4} by
constructing a discrete Lyapunov function. This construction is, of course,
deeply inspired by the underlying geometry of the phase space symmetries and
dynamics. It is also crucially linked to the fact that discrete Lyapunov
functions are only required to be continuous.

Although our construction depends on the symmetry of the dynamical
system, the fact that the equal clock problem  is modeled by Morse-Smale diffeomorphisms implies that the
dynamics is structurally stable with deep consequences in real world applications, since the previous, but essential results on equal clocks would be of reduced effect, since there are no equal clocks in real world applications besides quantum dynamics. 

We aim to extend the study presented in this article to the case of interacting oscillators with nearly multiple integer frequencies, as in \cite{HOPe2024}, where the research was carried out for two interacting clocks. Another line of research is to extend these results to a line of $N$ oscillators with nearest-neighbour interactions. 

Related to the present article, we conclude this paper with the conjecture, supported by strong numerical evidence, that it is possible to construct explicitly a complete Lyapunov function for the map $G$ on the torus.

\begin{conjecture}  Define the continuous function $\mathcal{L}:  \mathbb{T}^{2} \to  \mathbb{R}$
\[
\mathcal{L}(x,y) =
\begin{cases} 
\left(x-\frac{2 \pi}{3} \right)^2
+\left(y-\frac{4 \pi}{3} \right)^2
-\left(x-\frac{2 \pi}{3} \right) \left(y-\frac{4 \pi}{3} \right), & \text{if } y \geq x, \\
\left(y-\frac{2 \pi}{3} \right)^2
+\left(x-\frac{4 \pi}{3} \right)^2
-\left(y-\frac{2 \pi}{3} \right) 
\left( x-\frac{4 \pi}{3} \right)  & \text{if } y < x,
\end{cases}
\]
where we consider $\left( x, y \right) \in \left[ 0, 2\pi \right[ \times \left[ 0, 2\pi \right[$. 

This function $\mathcal{L}$ is a {\em complete} Lyapunov function in $\mathbb{T}^2$ for the map $G$ in the sense of Conley \cite{conley1978,norton1995}, possessing the required properties 
on the sets $S$ and $R$ described in Theorems~\ref{thm:_LyapunovG1} and \ref{thm:_LyapunovG2},
which are the (open) basins of attraction of the corresponding asymptotically stable fixed points.

\end{conjecture}

\subsection*{Acknowledgements}

The author Jorge Buescu was partially supported by  Funda\c{c}\~ao para a Ci\^encia e a
Tecnologia, UIDB/04561/2025.

The author Henrique M. Oliveira was partially supported by  Funda\c{c}\~ao para a Ci\^encia e a 
Tecnologia, UIDB/04459/2025 and UIDP/04459/2025.

\subsection*{Data availability}
Not applicable. The proofs and calculations were presented in the current article. Any queries can be addressed to the corresponding author.

\subsection*{Disclosure of interest}
The authors report no conflict of interest.

\bibliographystyle{abbrv}
\bibliography{BibloH2}

\begin{thebibliography}{10}

\bibitem{adler1946study}
R.~Adler.
\newblock A study of locking phenomena in oscillators.
\newblock {\em Proceedings of the IRE}, 34(6):351--357, 1946.

\bibitem{andronovtheory}
A.~A. Andronov, A.~G. Maier, I.~I. Gordon, and E.~A. Leontovich.
\newblock {\em Theory of bifurcations of dynamic systems on a plane}.
\newblock NASA Technical Translations Collection mir-titles, originally 1967,
  Washington DC, 1971.

\bibitem{Andronov1937}
A.~A. Andronov and L.~S. Pontrjagin.
\newblock Robust systems (in {R}ussian).
\newblock {\em DAN}, 14(5), 1937.

\bibitem{And}
A.~A. Andronov, A.~A. Vitt, and S.~E. Khaikin.
\newblock {\em Theory of Oscillators}.
\newblock Pergammon Press, Oxford, New York, 1959/1963/1966.

\bibitem{ASHWIN1994126}
P.~Ashwin, J.~Buescu, and I.~Stewart.
\newblock Bubbling of attractors and synchronization of chaotic oscillators.
\newblock {\em Physics Letters A}, 193(2):126--139, 1994.

\bibitem{ashwin1996attractor}
P.~Ashwin, J.~Buescu, and I.~Stewart.
\newblock From attractor to chaotic saddle: {A} tale of transverse instability.
\newblock {\em Nonlinearity}, 9(3):703--737, 1996.

\bibitem{auslander}
M.~Auslander and D.~Buchsbaum.
\newblock {\em Groups, rings, modules}.
\newblock Harper and Row/Dover, New York, 1974/2014.

\bibitem{Baigent2023}
S.~Baigent, S.~E. Z.~Hou, E.~C. Balreira, and R.~Luís.
\newblock A global picture for the planar {R}icker map: convergence to fixed
  points and identification of the stable/unstable manifolds.
\newblock {\em Journal of Difference Equations and Applications},
  29(5):575--591, 2023.

\bibitem{bertram1960}
R.~K.-J. Bertram and R.~Kalman.
\newblock Control systems analysis and design via the second method of
  {L}yapunov.
\newblock {\em Trans. ASME, D}, 82:394--400, 1960.

\bibitem{buescu2012exotic}
J.~Buescu.
\newblock {\em Exotic attractors. {From} {Liapunov} stability to riddled
  basins}, volume 153 of {\em Prog. Math.}
\newblock Basel: Birkh{\"a}user, 1997.

\bibitem{BAO2024b}
J.~Buescu, E.~D’Aniello, and H.~M. Oliveira.
\newblock Huygens synchronization of three aligned clocks.
\newblock {\em Nonlinear Dynamics}, 113(6):5457--5470, 2025.

\bibitem{BAO2025}
J.~Buescu, E.~D’Aniello, and H.~M. Oliveira.
\newblock A {L}yapunov function for a synchronisation diffeomorphism of three
  clocks.
\newblock {\em arXiv:2502.15371}, 2025.

\bibitem{conley1978}
C.~C. Conley.
\newblock {\em Isolated invariant sets and the Morse index}.
\newblock Number~38. American Mathematical Soc., 1978.

\bibitem{EH2}
E.~D'Aniello and H.~M. Oliveira.
\newblock Huygens synchronisation of three clocks equidistant from each other.
\newblock {\em Nonlinear Dynamics}, 112(5):3303--3317, 2024.

\bibitem{field1970}
M.~Field.
\newblock Equivariant dynamical systems.
\newblock {\em Bulletin of the American Mathematical Society},
  76(6):1314--1318, 1970.

\bibitem{field1980}
M.~Field.
\newblock Equivariant dynamical systems.
\newblock {\em Transactions of the American Mathematical Society},
  259(1):185--205, 1980.

\bibitem{giesl2015}
P.~Giesl and S.~Hafstein.
\newblock Review on computational methods for {L}yapunov functions.
\newblock {\em Discrete and Continuous Dynamical Systems-B}, 20(8):2291--2331,
  2015.

\bibitem{golubitsky2015}
M.~Golubitsky and I.~Stewart.
\newblock Recent advances in symmetric and network dynamics.
\newblock {\em Chaos: An Interdisciplinary Journal of Nonlinear Science},
  25(9), 2015.

\bibitem{Gu1983}
J.~Guckenheimer and P.~Holmes.
\newblock {\em Nonlinear oscillations, dynamical systems, and bifurcations of
  vector fields}, volume~42.
\newblock Springer, Berlin, 2013.

\bibitem{Hale1984}
J.~K. Hale, L.~T. Magalh{\~a}es, and W.~M. Oliva.
\newblock {\em Stability of {M}orse-{S}male Maps}, pages 111--139.
\newblock Springer New York, New York, NY, 1984.

\bibitem{hirsch2012differential}
M.~W. Hirsch, S.~Smale, and R.~L. Devaney.
\newblock {\em Differential equations, dynamical systems, and an introduction
  to chaos}.
\newblock Academic press, 2012.

\bibitem{Huy}
C.~Huygens.
\newblock {\em Letters to de Sluse, Constantyn Huygens, (letters; no. 1333 of
  24 February 1665, no. 1335 of 26 February 1665, no. 1345 of 6 March 1665)}.
\newblock Societe Hollandaise Des Sciences, Martinus Nijho, La Haye, 1895.

\bibitem{lasalle1976stability}
J.~P. LaSalle.
\newblock {\em The stability of dynamical systems}, volume~25.
\newblock Siam, 1976.

\bibitem{lyapunov1892}
A.~M. Lyapunov.
\newblock The general problem of the stability of motion, {E}nglish trans.
\newblock {\em International Journal of Control}, 55(3):531--534, originally
  1892, transl. 1992.

\bibitem{norton1995}
D.~E. Norton.
\newblock The fundamental theorem of dynamical systems.
\newblock {\em Commentationes Mathematicae Universitatis Carolinae},
  36(3):585--597, 1995.

\bibitem{OlMe}
H.~M. Oliveira and L.~V. Melo.
\newblock Huygens synchronization of two clocks.
\newblock {\em Scientific Reports}, 5(11548):1--12, 2015.
\newblock doi: 10.1038/srep11548.

\bibitem{HOPe2024}
H.~M. Oliveira and S.~Perestrelo.
\newblock The slow clock is the master.
\newblock {\em Nonlinear Dynamics}, 42(2):715--737, 2024.
\newblock doi: 10.12775/TMNA.2016.031.

\bibitem{PALIS1969}
J.~Palis.
\newblock On {M}orse-{S}male dynamical systems.
\newblock {\em Topology}, 8(4):385--404, 1969.

\bibitem{Palis1982}
J.~Palis and W.~de~Melo.
\newblock {\em Genericity and Stability of {M}orse-{S}male Vector Fields},
  pages 115--188.
\newblock Springer US, New York, NY, 1982.

\bibitem{palis1970structural}
J.~Palis and S.~Smale.
\newblock Structural stability theorems.
\newblock {\em Global Analysis}, pages 223--231, 1970.

\bibitem{palis2000structural}
J.~Palis and S.~Smale.
\newblock Structural stability theorems.
\newblock In {\em The Collected Papers of Stephen Smale: Volume 2}, pages
  739--747. World Scientific, 2000.

\bibitem{Peixoto1959}
M.~M. Peixoto.
\newblock On structural stability.
\newblock {\em Annals of Mathematics}, 69(1):199--222, 1959.

\bibitem{Peixoto1962}
M.~M. Peixoto.
\newblock Structural stability on two-dimensional manifolds.
\newblock {\em Topology}, 1(2):101--120, 1962.

\bibitem{Pit}
A.~Pikovsky, M.~Rosenblum, and J.~Kurths.
\newblock {\em Synchronization: A Universal Concept in Nonlinear Sciences},
  volume~12 of {\em Cambridge Nonlinear Science Series}.
\newblock Cambridge University Press, Cambridge, 1 edition, 5 2003.

\bibitem{sassano2013dynamic}
M.~Sassano and A.~Astolfi.
\newblock Dynamic {L}yapunov functions.
\newblock {\em Automatica}, 49(4):1058--1067, 2013.

\bibitem{Smale1960}
S.~Smale.
\newblock Morse inequalities for a dynamical system.
\newblock {\em Bulletin of the American Mathematical Society}, 66:43--49, 1960.

\bibitem{smale1967}
S.~Smale.
\newblock Differentiable dynamical systems.
\newblock {\em Bulletin of the American Mathematical Society}, 73(6):747--817,
  1967.

\bibitem{Stewart2003SymmetryGA}
I.~N. Stewart, M.~Golubitsky, and M.~Pivato.
\newblock Symmetry groupoids and patterns of synchrony in coupled cell
  networks.
\newblock {\em SIAM J. Appl. Dyn. Syst.}, 2:609--646, 2003.

\bibitem{strogatz2004}
S.~Strogatz.
\newblock {\em Sync: The emerging science of spontaneous order}.
\newblock Penguin UK, London, UK, 2004.

\end{thebibliography}

\end{document}